\documentclass[12pt]{article}
\usepackage{amssymb}
\usepackage{amsfonts}
\usepackage{amsmath}

\newtheorem{theorem}{Theorem}[section]

\newtheorem{case}{Case}[section]
\newtheorem{claim}{Claim}[section]

\newtheorem{corollary}{Corollary}[section]

\newtheorem{definition}{Definition}[section]

\newtheorem{lemma}{Lemma}[section]
\newtheorem{notation}{Notation}[section]

\newtheorem{proposition}{Proposition}[section]
\newtheorem{remark}{Remark}[section]

\newenvironment{proof}[1][Proof]{\noindent\textbf{#1.} }{\ \rule{0.5em}{0.5em}}
\input{tcilatex}
\begin{document}

$\varphi -${\Large Contractive parent-child infinite IFSs and orbital }$%
\varphi -${\Large contractive infinite IFSs}

\bigskip

Alexandru Mihail$^{\text{a}}$ and Irina Savu$^{\text{b}}$

\bigskip

$^{\text{a}}${\small University of Bucharest, Faculty of Mathematics and
Computer Science, Str. Academiei 14, 010014 Bucharest, Romania,
mihail\_alex@yahoo.com;}

$^{\text{b}}${\small University Politehnica of Bucharest, Faculty of Applied
Sciences, Str. Splaiul Independen\c{t}ei 313, 060042 Bucharest, Romania,
maria\_irina.savu@upb.ro, ORCID iD https://orcid.org/0000-0003-3647-3152. }

\begin{quotation}
{\small ABSTRACT}

{\small In this paper we introduce the notions of }$\varphi ${\small %
-contractive parent-child infinite iterated function system (pcIIFS) and
orbital }$\varphi ${\small -contractive infinite iterated function system
(oIIFS) and we prove that the corresponding fractal\ operator is weakly
Picard. The corresponding notions of shift space, canonical projection and
their properties are also treated. \ }


{\small MATHEMATICS SUBJECT CLASSIFICATION (2010)}

{\small 28A80, 37C70, 54H25.}
\end{quotation}

\section{\protect\Large Introduction}

\ \ \ \ \ The concept of iterated function system (IFS for short) was
introduced forty years ago by J. Hutchinson in \cite{Hutch}, popularized by
Michael Barnsley in \cite{Barnsley} and it represents one of the most
important methods of constructing fractals. In the last years there have
been considered many generalizations of this concept. A first direction of
generalization consists of using weaker contractivity conditions. Along
these lines of research let us mention some papers. For example, in \cite%
{Ioana} L. Ioana and A. Mihail introduced and studied IFSs consisting of $%
\varphi $-contractions and in \cite{Savu} I. Savu studied IFSs consisting of
continuous functions satisfying Banach's orbital condition. In \cite{Ge1} F.
Georgescu generalized the concept of IFS consisting of convex contractions
and in \cite{Sec} N. A. Secelean introduced a new type of IFS, namely IFS
consisting of $F$-contractions. The fractal operator associated to IFSs with
weaker contractivity conditions could have the same properties as the
component functions (see \cite{Str}) or not (see \cite{Miculescu}). In this
case, it may appear some difficulties. For example, in \cite{VanDung} N. Van
Dung and A. Petru\c{s}el pointed out the problems in providing some results
of IFSs consisting of Kannan maps, Reich maps and Chatterjea type maps.

A second way to generalize the notion of IFS was to consider systems
consisting of an arbitrary number (finite or infinite) of functions. For
example, in \cite{Fernau} H. Fernau studied infinite IFSs, in \cite{Luk} G.
Gw\'{o}\'{z}d\'{z}-Lukowska and J. Jachymski presented the
Hutchinson-Barnsley theory for infinite IFSs and in \cite{Dum} D. Dumitru
studied arcwise connected attractors of infinite IFSs. Also, I. Jaksztas in 
\cite{Jak} studied the infinite IFSs depending on a parameter and F.
Mendivil in \cite{Mendivil} constructed a generalization of IFS with
probabilities to infinitely many maps. The infinite IFSs were also studied
in \cite{Chi}, \cite{Hille}, \cite{Mauldin} and \cite{MiculMihail2}.

Another way to generalize the IFSs was to change the structure of component
functions or the structure of space. For example, R. D. Mauldin and M. Urba%
\'{n}ski in \cite{Urb} studied graph directed Markov systems and in \cite%
{MiculUrz} R. Miculescu and S. Urziceanu studied the canonical projection of
generalized IFSs.

Related to the concept of IFS is the notion of shift (or code) space. The
shift space of an iterated function system and the address of the points
lying on the attractor of the IFS are very good tools to get a more precise
description of the invariant dynamics of the IFS and of the topological
properties of the attractor. For example, in \cite{Dumitru} D. Dumitru
studied the topological properties of the attractors of IFS and in \cite%
{Strobin} F. Strobin studied this problem in the framework of generalized
iterated function systems.

In this paper we use the notion of parent-child contractivity condition to
define the notion of $\varphi $-contractive parent-child infinite iterated
function system (pcIIFS). The parent-child contractivity condition was also
used in \cite{Zaharopol} by R. Zaharopol\ to study IFS with probabilities.
Another notion introduced in this paper is that of orbital $\varphi $%
-contractive infinite iterated function system (oIIFS).

In the first part of the main results we study the fractal operator
associated to a pcIIFS and we prove that it is weakly Picard. Also, we
construct the canonical projection and we study its properties. We define a
continuous function $\Theta $ (which is uniformly continuous on bounded
sets) which describes the dynamics of an infinite iterated function system
(IIFS for short) better than the canonical projection. Using the function $%
\Theta $ we obtain the canonical projection and we define an extended
canonical projection, $\pi ^{t}$.

If we consider a pcIIFS $\mathcal{S}=\left( \left( X,d\right) ,\left(
f_{i}\right) _{i\in I}\right) $ such that the fractal operator associated to 
$\mathcal{S}$ is a Picard operator, then as $S_{\Lambda \left( I\right)
}=\left( \Lambda \left( I\right) ,\left( F_{i}\right) _{i\in I}\right) $ is
a universal model for the pcIIFS $S$ restricted to its fixed point, we have
that the system $S_{\Lambda ^{t}\left( I\right) }=(\Lambda ^{t}\left(
I\right) \times X,\left( F_{i}^{t}\right) _{i\in I})$ is a universal model
for the pcIIFS $S$ (see Remark \ref{rem}).

The second part of the main results is dedicated to the study of oIIFSs. We
study similar properties with those presented in the first part.

\section{\protect\Large Preliminaries}

\textbf{Notations and terminology}

\bigskip

Given a set $X$, a function $f:X\rightarrow X$ and $n\in 
\mathbb{N}
^{\ast },$ by $f^{[n]}$ we mean $f\circ f\circ ...\circ f$ \ for $n$ times.
Also, by $Id_{X}:X\rightarrow X$ we mean the function defined by $%
Id_{X}\left( x\right) =x$ for every $x\in X$.

Given a metric space $\left( X,d\right) ,$ by:

- $diam\left( A\right) $ we mean the diameter of the subset $A$ of $X;$

- $P_{b}\left( X\right) $ we mean the set of non-empty bounded subsets of $%
X; $

- $P_{cl,b}\left( X\right) $ we mean the set of non-empty closed and bounded
subsets of $X;$

- a weakly Picard operator we mean a function $f:X\rightarrow X$ having the
property that, for every $x\in X,$ the sequence $\left( f^{[n]}\left(
x\right) \right) _{n\in 
\mathbb{N}
}$ is convergent to a fixed point of $f$;

- a Picard operator we mean a function $f:X\rightarrow X$ having the
property that the sequence $\left( f^{[n]}\left( x\right) \right) _{n\in 
\mathbb{N}
}$ is convergent to the unique fixed point of $f$, for all $x\in X$;

- a continuous weakly Picard operator we mean a function $f:X\rightarrow X$
which is a weakly Picard operator and the function $x\rightarrow %
\displaystyle\lim_{n\rightarrow \infty }f^{n}\left( x\right) $ for all $x\in
X$ is continuous.

Let $\left( X,d_{X}\right) $ and $\left( Y,d_{Y}\right) $ be two metric
spaces. By

- $\mathcal{C}\left( X,Y\right) $ we mean the set of continuous functions
from $X$ to $Y$;

- $\mathcal{C}_{b}\left( X,Y\right) $ we mean the set of continuous and
bounded functions from $X$ to $Y$;

- $d_{\max }$ we mean the metric on $X\times Y$ defined by $d_{\max }\left(
\left( x_{1},y_{1}\right) ,\left( x_{2},y_{2}\right) \right) =\max \left\{
d_{X}\left( x_{1},x_{2}\right) ,d_{Y}\left( y_{1},y_{2}\right) \right\} $
for all $\left( x_{1},y_{1}\right) ,\left( x_{2},y_{2}\right) \in X\times Y$.

Given a metric space $(X,d)$ and two functions $f,g:X\rightarrow X$, by $%
d_{u}\left( f,g\right) $ we mean the uniform distance between $f$ and $g$,
namely 
\begin{equation*}
d_{u}\left( f,g\right) =\sup_{x\in X}d\left( f\left( x\right) ,g\left(
x\right) \right) \text{.}
\end{equation*}

\begin{definition}
Let $\left( X,d_{X}\right) $ and $\left( Y,d_{Y}\right) $ be two metric
spaces. A family of functions $\left( f_{i}\right) _{i\in I}$ is said to be

1) bounded if the set $\cup _{i\in I}f_{i}\left( B\right) $ $\in P_{b}\left(
X\right) $ for every $B\in $ $P_{b}\left( X\right) $,

2) equal uniformly continuous if for every $\varepsilon >0$ there exists $%
\delta _{\varepsilon }>0$ such that for all $x,y\in X$ with $d_{X}\left(
x,y\right) <\delta _{\varepsilon }$ we have $d_{Y}\left( f_{i}\left(
x\right) ,f_{i}\left( y\right) \right) <\varepsilon $, for all $i\in I$.
\end{definition}

\begin{definition}
For a metric space $\left( X,d\right) $ we consider on $P_{b}\left( X\right) 
$ the generalized Hausdorff-Pompeiu pseudometric $h:P_{b}\left( X\right)
\times P_{b}\left( X\right) \rightarrow \lbrack 0,\infty )$ defined by%
\begin{equation*}
h\left( A,B\right) =\max \left\{ d\left( A,B\right) ,d\left( B,A\right)
\right\}
\end{equation*}%
for all $A,B\in P_{b}\left( X\right) ,$ where $d\left( A,B\right) =%
\displaystyle\sup_{x\in A}\displaystyle\inf_{y\in B}d\left( x,y\right) $ and 
$\ d\left( B,A\right) =\displaystyle\sup_{x\in B}\displaystyle\inf_{y\in
A}d\left( x,y\right) $.
\end{definition}

\begin{definition}
The restriction of $h$ to $P_{cl,b}\left( X\right) $ is called the
Hausdorff-Pompeiu metric and it is also denoted by $h$.
\end{definition}

\begin{notation}
For a set $A\in P_{b}\left( X\right) $ and a number $r>0$, by $B\left[ A,r%
\right] $ we mean the set $\left\{ x\in X\text{ }|\text{ }h\left( x,A\right)
\leq r\right\} $.
\end{notation}

\textbf{Results regarding the Hausdorff-Pompeiu semidistance:}

\begin{proposition}
\label{h}[see \cite{SeceleanN}] For a metric space $\left( X,d\right) $ the
assertions below hold:

1) If $H$, $K\in $ $P_{b}\left( X\right) $, then%
\begin{equation}
h\left( H,K\right) =h\left( \overline{H},\overline{K}\right) \text{;}
\label{hpeinch}
\end{equation}

2) If $\left( H_{i}\right) _{i\in I\text{ }}$and $\left( K_{i}\right) _{i\in
I}$ are families of elements from $P_{b}\left( X\right) $ such that $%
\displaystyle\cup _{i\in I}H_{i}\in P_{b}\left( X\right) $ and $\displaystyle%
\cup _{i\in I}K_{i}\in P_{b}\left( X\right) $, then 
\begin{equation}
h\left( \displaystyle\cup _{i\in I}H_{i},\displaystyle\cup _{i\in
I}K_{i}\right) \leq \sup_{i\in I}h\left( H_{i},K_{i}\right) \text{;}
\label{sup}
\end{equation}

3) If $f:X\rightarrow X$ is a uniformly continuous function, $A\in
P_{b}\left( X\right) $ and $\left( A_{n}\right) _{n\in 
\mathbb{N}
}\subset P_{b}\left( X\right) $ such that $\displaystyle\lim_{n\rightarrow
\infty }h\left( A_{n},A\right) =0$, then $\displaystyle\lim_{n\rightarrow
\infty }h\left( f\left( A_{n}\right) ,f\left( A\right) \right) =0$.
\end{proposition}

\begin{proposition}
\label{complet}[see \cite{SeceleanN}] \ If the metric space $\left(
X,d\right) $ is complete, then the metric space $\left( P_{cl,b}\left(
X\right) ,h\right) $ is complete.
\end{proposition}

\textbf{Notations and terminology for the shift space}

\bigskip

$%
\mathbb{N}
$ denotes the natural numbers, $%
\mathbb{N}
^{\ast }=%
\mathbb{N}
\diagdown \left\{ 0\right\} $ and $%
\mathbb{N}
_{n}^{\ast }=\left\{ 1,2,...,n\right\} $, where $n\in 
\mathbb{N}
^{\ast }$. Given two sets $A$ and $B$, by $B^{A}$ we mean the set of all
functions from $A$ to $B$.

For a set $I$, by $\Lambda \left( I\right) $ we mean the set $I^{%
\mathbb{N}
^{\ast }}$. The elements of $\Lambda \left( I\right) $ can be written as
infinite words, namely $\omega =\omega _{1}\omega _{2}...\omega _{n}...$ .
For $\omega \in \Lambda \left( I\right) $ and $n\in 
\mathbb{N}
^{\ast }$ by $[\omega ]_{n}$ we mean the word formed with the first $n$
letters from $\omega $. For $\omega \in \Lambda _{m}\left( I\right) $ and $%
n\in 
\mathbb{N}
^{\ast }$, by $[\omega ]_{n}$ we mean the word formed with the first $n$
letters from $\omega $ if $m\geq n$, or the word $\omega $ if $m\leq n$.

By $\Lambda _{n}\left( I\right) $ we mean the set $I^{%
\mathbb{N}
_{n}^{\ast }}$. The elements of $\Lambda _{n}\left( I\right) $ are finite
words with $n$ letters: $\omega =\omega _{1}\omega _{2}...\omega _{n}$. In
this case, $n$ is called the length of $\omega $ and it is denoted by $%
|\omega |$. For two words $\alpha \in \Lambda _{n}\left( I\right) $ and $%
\beta \in \Lambda _{m}\left( I\right) $ or $\beta \in \Lambda \left(
I\right) $, by $\alpha \beta $ we mean the concatenation of $\alpha $ and $%
\beta $, i.e. $\alpha \beta =\alpha _{1}\alpha _{2}...\alpha _{n}\beta
_{1}\beta _{2}...\beta _{m}$ and $\alpha \beta =\alpha _{1}\alpha
_{2}...\alpha _{n}\beta _{1}\beta _{2}...\beta _{m}\beta _{m+1}...$
respectively.

For a family of functions $\left( f_{i}\right) _{i\in I},$ where $%
f_{i}:X\rightarrow X$ and $\omega =\omega _{1}\omega _{2}...\omega _{n}\in
\Lambda _{n}\left( I\right) $, we use the following notation: $f_{\omega
}=f_{\omega _{1}}\circ ...\circ f_{\omega _{n}}$. For a set $B$ $\subset X$
and $\omega \in \Lambda _{n}\left( I\right) $ we use the notation $B_{\omega
}=f_{\omega }\left( B\right) $. For $y,z\in B$, we say that $z$ is a child
of $y$ (or $y$ is a parent of $z$) if there exist $n\in 
\mathbb{N}
^{\ast }$, $\omega _{1},\omega _{2},...,\omega _{n+1}\in I$ and $x\in B$
such that $y=f_{\omega _{1}}\circ ...\circ f_{\omega _{n}}\left( x\right) $
and $z=f_{\omega _{1}}\circ ...\circ f_{\omega _{n}}\circ f_{\omega
_{n+1}}\left( x\right) $.

By $\Lambda ^{\ast }\left( I\right) $ we mean the set of all finite words, $%
\Lambda ^{\ast }\left( I\right) =\left( \displaystyle\cup _{n\in 
\mathbb{N}
^{\ast }}\Lambda _{n}\left( I\right) \right) \cup \left\{ \lambda \right\} $%
, where $\lambda $ is the empty word. By $\Lambda ^{t}\left( I\right) $ we
mean the set of all words with letters from $I\,$, namely the set $\Lambda
^{\ast }\left( I\right) \cup \Lambda \left( I\right) $.

\bigskip

For a fixed element $\tau \notin I$, we denote by $\widetilde{I}=I\cup
\left\{ \tau \right\} $. Let us consider $c\in \lbrack 0,1)$. We define a
function $d_{c}:\Lambda (\widetilde{I})\times \Lambda (\widetilde{I}%
)\rightarrow \lbrack 0,\infty )$ by 
\begin{equation*}
d_{c}\left( \alpha ,\beta \right) =\sum_{n\geq 1}c^{n}d\left( \alpha
_{n},\beta _{n}\right) ,
\end{equation*}%
for all $\alpha ,\beta \in \Lambda (\widetilde{I})$, where by $\alpha _{n}$
we mean the letter on position $n$ in $\alpha $ and $d\left( \alpha
_{n},\beta _{n}\right) =\left\{ 
\begin{array}{c}
1\text{, if }\alpha _{n}=\beta _{n} \\ 
0\text{, if }\alpha _{n}\neq \beta _{n}%
\end{array}%
\right. $ for all $n\in 
\mathbb{N}
^{\ast }$.

\begin{remark}
$\Lambda ^{t}\left( I\right) $ can be seen as a subset of $\Lambda (%
\widetilde{I})$, by defining the injective function $\iota :\Lambda
^{t}\left( I\right) \rightarrow \Lambda (\widetilde{I})$, $\iota \left(
\alpha \right) =\left\{ 
\begin{array}{c}
\alpha \text{, if }\alpha \in \Lambda \left( I\right) \\ 
\alpha \tau \tau ...\tau ...\text{ if }\alpha \in \Lambda ^{\ast }\left(
I\right) \text{.} \\ 
\tau \tau ...\tau ...\text{ if }\alpha =\lambda%
\end{array}%
\right. $
\end{remark}

\begin{remark}
$\left( \Lambda ^{t}\left( I\right) \text{, }d_{c}\right) $ is a complete
metric space, $\Lambda \left( I\right) $ is a closed subset of $\Lambda
^{t}\left( I\right) $ and $\Lambda ^{\ast }\left( I\right) $ contains only
isolated points.
\end{remark}

\textbf{Infinite iterated function systems}

\begin{definition}
\label{compfunc}Let $\left( X,d\right) $ be a metric space. A function $%
\varphi :[0,\infty )\rightarrow \lbrack 0,\infty )$ is called

1) comparison function if

i) $\varphi (r)<r$ for all $r>0$,

ii) $\varphi $ is an increasing function on $[0,\infty )$;

2) summable comparison function if $\varphi $ is a comparison function and $%
\displaystyle\sum_{n=0}^{\infty }\varphi ^{n}\left( r\right) $ is convergent
for every $r>0$;

3) right continuous function if $\displaystyle\lim_{\underset{r>r_{0}}{%
r\rightarrow r_{0}}}\varphi \left( r\right) =\varphi \left( r_{0}\right) $
for all $r_{0}\in \lbrack 0,\infty )$.
\end{definition}

\begin{remark}
If $\varphi :[0,\infty )\rightarrow \lbrack 0,\infty )$ is a summable or
right continuous comparison function, then%
\begin{equation*}
\varphi \left( 0\right) =0
\end{equation*}%
and%
\begin{equation}
\lim_{n\rightarrow \infty }\varphi ^{n}\left( r\right) =0  \label{filan}
\end{equation}%
for every $r>0$.
\end{remark}

\begin{definition}
\label{ficontr}Let $\left( X,d\right) $ be a metric space and $\varphi
:[0,\infty )\rightarrow \lbrack 0,\infty )$ a comparison function. A
function $f:X\rightarrow X$ is called $\varphi $-contraction if 
\begin{equation*}
d\left( f\left( x\right) ,f\left( y\right) \right) \leq \varphi \left(
d\left( x,y\right) \right)
\end{equation*}%
for every $x,y\in X$.
\end{definition}

\begin{theorem}
\label{thmficontr}Let $\left( X,d\right) $ be a complete metric space, $%
\varphi :[0,\infty )\rightarrow \lbrack 0,\infty )$ a right continuous
comparison function and $f:X\rightarrow X$ a $\varphi $-contraction. Then $f$
has a unique fixed point. If $\eta $ is the fixed point of $f$, then for
every $x_{0}\in X$\ the sequence $\left( f^{n}\left( x_{0}\right) \right)
_{n\in 
\mathbb{N}
}$ is convergent to $\eta $ and 
\begin{equation*}
d\left( f^{n}\left( x_{0}\right) ,\eta \right) \leq \varphi ^{n}\left(
d\left( x_{0},\eta \right) \right)
\end{equation*}%
for all $n\in 
\mathbb{N}
$.
\end{theorem}

\begin{definition}
\label{iifs1}Let $\left( X,d\right) $ be a complete metric space and $\left(
f_{i}\right) _{i\in I}$ an infinite family of functions, where $%
f_{i}:X\rightarrow X$. The pair denoted by $\mathcal{S}=\left( \left(
X,d\right) ,\left( f_{i}\right) _{i\in I}\right) $ is called infinite
iterated function system (IIFS) if

i) $f_{i}:X\rightarrow X$ is a continuous function for every $i\in I$,

ii) the family $\left( f_{i}\right) _{i\in I}$ is equal uniformly continuous
on bounded sets, i.e. for every $B\in P_{b}\left( X\right) $ and every $%
\varepsilon >0$ there exists $\delta _{\varepsilon ,B}>0$ such that for all $%
x,y\in B$ with $d\left( x,y\right) <\delta _{\varepsilon ,B}$ we have $%
d\left( f_{i}\left( x\right) ,f_{i}\left( y\right) \right) <\varepsilon $,
for all $i\in I$,

iii) $\left( f_{i}\right) _{i\in I}$ is a bounded family of functions.

Given an IIFS\ $\mathcal{S}=\left( \left( X,d\right) ,\left( f_{i}\right)
_{i\in I}\right) $ we consider the fractal operator $F_{\mathcal{S}%
}:P_{b}\left( X\right) \rightarrow P_{b}\left( X\right) $ defined by 
\begin{equation*}
F_{\mathcal{S}}\left( B\right) =\overline{\cup _{i\in I}f_{i}\left( B\right) 
}
\end{equation*}%
for every $B\in P_{b}\left( X\right) $. The restriction of $F_{\mathcal{S}}$
to $P_{cl,b}\left( X\right) $ will still be denoted by $F_{\mathcal{S}}$.
\end{definition}

\begin{notation}
For a set $B\in P_{b}\left( X\right) $, by the orbit of $B$ we mean the set $%
\mathcal{O}\left( B\right) $ $=\displaystyle\cup _{n\in 
\mathbb{N}
}F_{\mathcal{S}}^{[n]}\left( B\right) $. If $B=\left\{ x\right\} $, we
denote its orbit by $\mathcal{O}\left( x\right) $.
\end{notation}

\begin{definition}
\label{iifs}Let $\mathcal{S}=\left( \left( X,d\right) ,\left( f_{i}\right)
_{i\in I}\right) $ be an IIFS. $\mathcal{S}$ is called

1) $\varphi $-contractive parent-child infinite iterated function system
(pcIIFS) if $\varphi :[0,\infty )\rightarrow \lbrack 0,\infty )$ is a
summable comparison function and 
\begin{equation}
d\left( f_{\omega }\left( x\right) ,f_{\omega i}\left( x\right) \right) \leq
\varphi ^{|\omega |}\left( d\left( x,f_{i}\left( x\right) \right) \right) 
\text{,}  \label{pc}
\end{equation}%
for every $i\in I$, $\omega \in \Lambda ^{\ast }\left( I\right) $ and $x\in
X $;

2) orbital $\varphi $-contractive infinite iterated function system (oIIFS)
if $\varphi :[0,\infty )\rightarrow \lbrack 0,\infty )$ is a
right-continuous comparison function and%
\begin{equation*}
d\left( f_{i}\left( y\right) ,f_{i}\left( z\right) \right) \leq \varphi
\left( d\left( y,z\right) \right)
\end{equation*}%
for every $i\in I$, $x\in X$ and $y,z\in \mathcal{O}\left( x\right) $.
\end{definition}

\begin{remark}
If $\mathcal{S}=\left( \left( X,d\right) ,\left( f_{i}\right) _{i\in
I}\right) $ is an oIIFS, then%
\begin{equation}
d\left( f_{\omega }\left( y\right) ,f_{\omega }\left( z\right) \right) \leq
\varphi ^{|\omega |}\left( d\left( y,z\right) \right) \text{,}  \label{o}
\end{equation}%
for every $\omega \in \Lambda ^{\ast }\left( I\right) $, $x\in X$ and $%
y,z\in \mathcal{O}\left( x\right) $.
\end{remark}

\begin{remark}
\label{MB}If $B\in P_{b}\left( X\right) $ and $\left( B_{n}\right) _{n\in 
\mathbb{N}
}$ $\subset P_{b}\left( X\right) $ such that $\displaystyle%
\lim_{n\rightarrow \infty }h\left( B_{n},B\right) =0$, we deduce that the
sequence $\left( B_{n}\right) _{n}$ is bounded. Therefore, there exists a
set $M\in P_{b}\left( X\right) $ such that $\left( \displaystyle\cup _{n\in 
\mathbb{N}
}B_{n}\right) \cup B\subset M$.

Let $\mathcal{S}=\left( \left( X,d\right) ,\left( f_{i}\right) _{i\in
I}\right) $ be an IIFS. Using ii) and iii) from Definition \ref{iifs1} we
deduce that the family $\left( f_{\alpha }\right) _{\alpha \in \Lambda
_{n}(I)}$ is equal uniformly continuous on $M$ and as a consequence $F_{%
\mathcal{S}}^{n}$ is uniformly continuous on $P_{b}(M)$ for every $n\in 
\mathbb{N}
$.
\end{remark}

\section{$\protect\varphi -${\protect\Large Contractive parent-child
infinite iterated function systems (pcIIFSs)}}

\begin{theorem}
Let $\mathcal{S}=\left( \left( X,d\right) ,\left( f_{i}\right) _{i\in
I}\right) $ be a pcIIFS. Then $F_{\mathcal{S}}$ is a weakly Picard operator.
More precisely, for every $B\in P_{cl,b}\left( X\right) $ there exists $%
A_{B}\in P_{cl,b}\left( X\right) $ such that $\displaystyle%
\lim_{n\rightarrow \infty }F_{\mathcal{S}}^{n}\left( B\right) =A_{B}$ and $%
F_{\mathcal{S}}\left( A_{B}\right) =A_{B}$. Moreover, 
\begin{equation}
h\left( F_{\mathcal{S}}^{n}\left( B\right) ,A_{B}\right) \leq \sum_{k\geq
n}\varphi ^{k}\left( diam\left( B\cup F_{\mathcal{S}}\left( B\right) \right)
\right)  \label{AB}
\end{equation}%
for all $n\in 
\mathbb{N}
$.
\end{theorem}

\begin{proof}
We consider $F_{\mathcal{S}}$ defined on $P_{cl,b}\left( X\right) $ but it
can be generalized to $P_{b}\left( X\right) $, as the image of a bounded set 
$B$ is a closed set.

Taking into consideration iii) from Definition \ref{iifs1} we deduce that
for each $x\in X$, $\displaystyle\sup_{i\in I}d\left( x,f_{i}\left( x\right)
\right) $ is finite. Let $B\in P_{cl,b}\left( X\right) $. We have 
\begin{equation}
\sup_{x\in B}\left( \sup_{i\in I}d\left( x,f_{i}\left( x\right) \right)
\right) \leq diam\left( B\cup F_{\mathcal{S}}\left( B\right) \right) \text{,}
\label{sup2}
\end{equation}%
so $\displaystyle\sup_{x\in B}\left( \displaystyle\sup_{i\in I}d\left(
x,f_{i}\left( x\right) \right) \right) $ is finite.

\begin{claim}
\begin{equation}
h\left( F_{\mathcal{S}}^{n}\left( \left\{ x\right\} \right) ,F_{\mathcal{S}%
}^{n+1}\left( \left\{ x\right\} \right) \right) \leq \varphi ^{n}\left(
\sup_{i\in I}d\left( x,f_{i}\left( x\right) \right) \right)  \label{cauchy1}
\end{equation}%
for all $n\in 
\mathbb{N}
$ and $x\in B$.
\end{claim}

Justification: We have%
\begin{equation*}
h\left( F_{\mathcal{S}}^{n}\left( \left\{ x\right\} \right) ,F_{\mathcal{S}%
}^{n+1}\left( \left\{ x\right\} \right) \right) =h\left( \cup _{\alpha \in
\Lambda _{n}\left( I\right) }\left\{ f_{\alpha }\left( x\right) \right\}
,\cup _{\alpha \in \Lambda _{n}\left( I\right) }f_{\alpha }\left( F_{%
\mathcal{S}}\left( \left\{ x\right\} \right) \right) \right)
\end{equation*}%
\begin{equation*}
\overset{(\ref{sup})}{\leq }\sup_{\alpha \in \Lambda _{n}\left( I\right)
}h\left( \left\{ f_{\alpha }\left( x\right) \right\} ,f_{\alpha }\left( F_{%
\mathcal{S}}\left( \left\{ x\right\} \right) \right) \right) =\sup_{\alpha
\in \Lambda _{n}\left( I\right) }h\left( \cup _{i\in I}\left\{ f_{\alpha
}\left( x\right) \right\} ,\cup _{i\in I}\left\{ f_{\alpha }\left(
f_{i}\left( x\right) \right) \right\} \right)
\end{equation*}%
\begin{equation*}
\overset{(\ref{sup})}{\leq }\sup_{\alpha \in \Lambda _{n}\left( I\right)
}\sup_{i\in I}h\left( \left\{ f_{\alpha }\left( x\right) \right\} ,\left\{
f_{\alpha }\left( f_{i}\left( x\right) \right) \right\} \right) \overset{%
\text{Definition }\ref{iifs}\text{ }1)}{\leq }\sup_{i\in I}\varphi
^{n}\left( d\left( x,f_{i}\left( x\right) \right) \right)
\end{equation*}%
\begin{equation*}
\overset{\text{Definition }\ref{compfunc}\text{ }1)}{\leq }\varphi
^{n}\left( \sup_{i\in I}d\left( x,f_{i}\left( x\right) \right) \right)
\end{equation*}%
for all $n\in 
\mathbb{N}
$. Thus,%
\begin{equation*}
h\left( F_{\mathcal{S}}^{n}\left( \left\{ x\right\} \right) ,F_{\mathcal{S}%
}^{n+1}\left( \left\{ x\right\} \right) \right) \leq \varphi ^{n}\left(
\sup_{i\in I}d\left( x,f_{i}\left( x\right) \right) \right)
\end{equation*}%
for all $n\in 
\mathbb{N}
$.

\begin{claim}
$\left( F_{\mathcal{S}}^{n}\left( B\right) \right) _{n\in 
\mathbb{N}
}$ is a Cauchy sequence. Moreover,%
\begin{equation}
h\left( F_{\mathcal{S}}^{m}\left( B\right) ,F_{\mathcal{S}}^{n}\left(
B\right) \right) \leq \sum_{k=m}^{n-1}\varphi ^{k}\left( diam\left( B\cup F_{%
\mathcal{S}}\left( B\right) \right) \right)  \label{fmn}
\end{equation}%
for all $B\in P_{cl,b}\left( X\right) $ and $m,n\in 
\mathbb{N}
$, $m<n$.\ 
\end{claim}

Justification: Let $B\in P_{cl,b}\left( X\right) $. We have 
\begin{equation*}
h\left( F_{\mathcal{S}}^{n}\left( B\right) ,F_{\mathcal{S}}^{n+1}\left(
B\right) \right) =h\left( \cup _{x\in B}F_{\mathcal{S}}^{n}\left( \left\{
x\right\} \right) ,\cup _{x\in B}F_{\mathcal{S}}^{n+1}\left( \left\{
x\right\} \right) \right)
\end{equation*}%
\begin{equation*}
\overset{\text{ }(\ref{sup})\text{ }}{\leq }\sup_{x\in B}h\left( F_{\mathcal{%
S}}^{n}\left( \left\{ x\right\} \right) ,F_{\mathcal{S}}^{n+1}\left( \left\{
x\right\} \right) \right) \overset{(\ref{cauchy1})}{\leq }\sup_{x\in
B}\varphi ^{n}\left( \sup_{i\in I}d\left( x,f_{i}\left( x\right) \right)
\right)
\end{equation*}%
\begin{equation*}
\overset{\text{Definition }\ref{compfunc}\text{ }1)}{\leq }\varphi
^{n}\left( \sup_{x\in B}\sup_{i\in I}d\left( x,f_{i}\left( x\right) \right)
\right) \overset{\text{ }(\ref{sup2})\text{ }}{\leq }\varphi ^{n}\left(
diam\left( B\cup F_{\mathcal{S}}\left( B\right) \right) \right)
\end{equation*}%
for each $n\in 
\mathbb{N}
$. So, 
\begin{equation}
h\left( F_{\mathcal{S}}^{n}\left( B\right) ,F_{\mathcal{S}}^{n+1}\left(
B\right) \right) \leq \varphi ^{n}\left( diam\left( B\cup F_{\mathcal{S}%
}\left( B\right) \right) \right)  \label{cauchy2}
\end{equation}%
for each $n\in 
\mathbb{N}
$. Let us now consider $m,n\in 
\mathbb{N}
$, $m<n$. Applying triangle inequality and (\ref{cauchy2}) we obtain%
\begin{equation*}
h\left( F_{\mathcal{S}}^{m}\left( B\right) ,F_{\mathcal{S}}^{n}\left(
B\right) \right) \leq \sum_{k=m}^{n-1}\varphi ^{k}\left( diam\left( B\cup F_{%
\mathcal{S}}\left( B\right) \right) \right)
\end{equation*}%
for all $m,n\in 
\mathbb{N}
$, $m<n$. Using the above inequality and 2) from Definition \ref{compfunc}
we deduce that $\left( F_{\mathcal{S}}^{n}\left( B\right) \right) _{n\in 
\mathbb{N}
}$ is a Cauchy sequence.

\begin{claim}
$\overline{\mathcal{O}\left( C\right) }$ is bounded for every $C\in
P_{b}\left( X\right) $.
\end{claim}

Justification: We have 
\begin{equation*}
h\left( \overline{\mathcal{O}\left( C\right) },C\right) \overset{(\ref%
{hpeinch})}{=}h\left( \cup _{n\in 
\mathbb{N}
}F_{\mathcal{S}}^{n}\left( C\right) ,\cup _{n\in 
\mathbb{N}
}C\right) \overset{(\ref{sup})}{\leq }\sup_{n\in 
\mathbb{N}
}h\left( F_{\mathcal{S}}^{n}\left( C\right) ,C\right)
\end{equation*}%
\begin{equation*}
\overset{(\ref{fmn})}{\leq }\sup_{n\in 
\mathbb{N}
}\sum_{k=0}^{n-1}\varphi ^{k}\left( diam\left( C\cup F_{\mathcal{S}}\left(
C\right) \right) \right) =\sum_{k=0}^{\infty }\varphi ^{k}\left( diam\left(
C\cup F_{\mathcal{S}}\left( C\right) \right) \right) \text{.}
\end{equation*}

We deduce that $\overline{\mathcal{O}\left( C\right) }\subset B\left[ C,%
\displaystyle\sum_{k=0}^{\infty }\varphi ^{k}\left( diam\left( C\cup F_{%
\mathcal{S}}\left( C\right) \right) \right) \right] $ and according to 2)
from Definition \ref{compfunc}, we have that $\overline{\mathcal{O}\left(
C\right) }$ is bounded.

\begin{claim}
For every $B\in P_{b}\left( X\right) $, there exists $A_{B}\in
P_{cl,b}\left( X\right) $ such that $\displaystyle\lim_{n\rightarrow \infty
}F_{\mathcal{S}}^{n}\left( B\right) =A_{B}$ and $F_{\mathcal{S}}\left(
A_{B}\right) =A_{B}$.
\end{claim}

Justification: We have seen that $\left( F_{\mathcal{S}}^{n}\left( \left(
B\right) \right) \right) _{n\in 
\mathbb{N}
}$ is a Cauchy sequence. Applying Proposition \ref{complet} and iii) from
Definition \ref{iifs1} we deduce that there exists $A_{B}\in P_{cl,b}\left(
X\right) $ such that $\displaystyle\lim_{n\rightarrow \infty }F_{\mathcal{S}%
}^{n}\left( B\right) =A_{B}$. Since the family of functions $\left(
f_{i}\right) _{i\in I}$ is equal uniformly continuous on bounded sets and $%
\overline{\mathcal{O}\left( B\right) }$ is bounded, it results that $F_{%
\mathcal{S}}$ is continuous on $\overline{\mathcal{O}\left( B\right) }$, so $%
\displaystyle\lim_{n\rightarrow \infty }F_{\mathcal{S}}\left( F_{\mathcal{S}%
}^{n}\left( B\right) \right) =F_{\mathcal{S}}\left( A_{B}\right) $. We
obtain that $F_{\mathcal{S}}\left( A_{B}\right) =A_{B}$. Hence, we conclude
that $F_{\mathcal{S}}$ is a weakly Picard operator. By passing to limit as $%
n\rightarrow \infty $ in (\ref{fmn}) we deduce 
\begin{equation*}
h\left( F_{\mathcal{S}}^{m}\left( B\right) ,A_{B}\right) \leq \sum_{k\geq
m}\varphi ^{k}\left( diam\left( B\cup F_{\mathcal{S}}\left( B\right) \right)
\right)
\end{equation*}%
for all $m\in 
\mathbb{N}
$.
\end{proof}

\begin{remark}
If we consider $B=\left\{ x\right\} $, then there exists $A_{x}\in
P_{cl,b}\left( X\right) $ such that $\displaystyle\lim_{n\rightarrow \infty
}F_{\mathcal{S}}^{n}\left( \left\{ x\right\} \right) =A_{\left\{ x\right\} }$%
. For the sake of simplicity we will denote $A_{\left\{ x\right\} }$ by $%
A_{x}$. In this case,%
\begin{equation}
h\left( F_{\mathcal{S}}^{n}\left( \left\{ x\right\} \right) ,A_{x}\right)
\leq \sum_{k\geq n}\varphi ^{k}\left( diam\left( \left\{ x\right\} \cup F_{%
\mathcal{S}}\left( \left\{ x\right\} \right) \right) \right)  \label{Ax1}
\end{equation}%
for every $n\in 
\mathbb{N}
$.
\end{remark}

\begin{proposition}
\label{ABx}Let $\mathcal{S}=\left( \left( X,d\right) ,\left( f_{i}\right)
_{i\in I}\right) $ be a pcIIFS. Then $\displaystyle A_{B}=\overline{\cup
_{x\in B}A_{x}}$ for every $B\in P_{cl,b}\left( X\right) $.
\end{proposition}

\begin{proof}
Let us consider $B\in P_{cl,b}\left( X\right) $. We have 
\begin{equation*}
h\left( F_{\mathcal{S}}^{n}\left( B\right) ,\overline{\cup _{x\in B}A_{x}}%
\right) \overset{(\ref{hpeinch})}{=}h\left( \cup _{x\in B}F_{\mathcal{S}%
}^{n}\left( \left\{ x\right\} \right) ,\cup _{x\in B}A_{x}\right)
\end{equation*}%
\begin{equation*}
\overset{(\ref{sup})}{\leq }\sup_{x\in B}h\left( F_{\mathcal{S}}^{n}\left(
\left\{ x\right\} \right) ,A_{x}\right) \overset{(\ref{Ax1})}{\leq }%
\sup_{x\in B}\sum_{k\geq n}\varphi ^{k}\left( diam\left( \left\{ x\right\}
\cup F_{\mathcal{S}}\left( \left\{ x\right\} \right) \right) \right)
\end{equation*}%
\begin{equation*}
\overset{\text{Definition }\ref{compfunc}\text{ }1)}{\leq }\sum_{k\geq
n}\varphi ^{k}\left( diam\left( B\cup F_{\mathcal{S}}\left( B\right) \right)
\right)
\end{equation*}%
for all $n\in 
\mathbb{N}
$. Hence, 
\begin{equation}
h\left( F_{\mathcal{S}}^{n}\left( B\right) ,\overline{\cup _{x\in B}A_{x}}%
\right) \leq \sum_{k\geq n}\varphi ^{k}\left( diam\left( B\cup F_{\mathcal{S}%
}\left( B\right) \right) \right)  \label{propA_B1}
\end{equation}%
for all $n\in 
\mathbb{N}
$. We deduce%
\begin{equation*}
h\left( A_{B},\overline{\cup _{x\in B}A_{x}}\right) \leq h\left( A_{B},F_{%
\mathcal{S}}^{n}\left( B\right) \right) +h\left( F_{\mathcal{S}}^{n}\left(
B\right) ,\overline{\cup _{x\in B}A_{x}}\right)
\end{equation*}%
\begin{equation*}
\overset{(\ref{AB}),(\ref{propA_B1})}{\leq }2\cdot \sum_{k\geq n}\varphi
^{k}\left( diam\left( B\cup F_{\mathcal{S}}\left( B\right) \right) \right) 
\text{,}
\end{equation*}%
so%
\begin{equation*}
h\left( A_{B},\overline{\cup _{x\in B}A_{x}}\right) \leq 2\cdot \sum_{k\geq
n}\varphi ^{k}\left( diam\left( B\cup F_{\mathcal{S}}\left( B\right) \right)
\right)
\end{equation*}%
for all $n\in 
\mathbb{N}
$. By passing to limit as $n\rightarrow \infty $ and applying 2) from
Definition \ref{compfunc} we conclude that $A_{B}=\overline{\displaystyle%
\cup _{x\in B}A_{x}}$.
\end{proof}

The following result shows that $F_{\mathcal{S}}$ is a continuous weakly
Picard operator.

\begin{proposition}
\label{Fcont}Let $\mathcal{S}=\left( \left( X,d\right) ,\left( f_{i}\right)
_{i\in I}\right) $ be a pcIIFS and let $F:P_{cl,b}\left( X\right)
\rightarrow P_{cl,b}\left( X\right) $ defined by $F\left( B\right) =A_{B}$,
for every $B\in P_{cl,b}\left( X\right) $. Then $F$ is continuous.
\end{proposition}

\begin{proof}
Let $\left( B_{n}\right) _{n\in 
\mathbb{N}
}\subset P_{cl,b}\left( X\right) $ and $B\in P_{cl,b}\left( X\right) $ such
that $\displaystyle\lim_{n\rightarrow \infty }h\left( B_{n},B\right) =0$.
Then,%
\begin{equation*}
h\left( F\left( B_{n}\right) ,F\left( B\right) \right) =h\left(
A_{B_{n}},A_{B}\right) \leq h\left( A_{B_{n}},F_{\mathcal{S}}^{m}\left(
B_{n}\right) \right)
\end{equation*}%
\begin{equation*}
+h\left( F_{\mathcal{S}}^{m}\left( B_{n}\right) ,F_{\mathcal{S}}^{m}\left(
B\right) \right) +h\left( F_{\mathcal{S}}^{m}\left( B\right) ,A_{B}\right)
\end{equation*}%
for each $m,n\in 
\mathbb{N}
$. Using (\ref{AB}) we have%
\begin{equation*}
h\left( A_{B_{n}},F_{\mathcal{S}}^{m}\left( B_{n}\right) \right) \leq
\sum_{k\geq m}\varphi ^{k}\left( diam\left( B_{n}\cup F_{\mathcal{S}}\left(
B_{n}\right) \right) \right)
\end{equation*}%
for each $m,n\in 
\mathbb{N}
$ and%
\begin{equation*}
h\left( F_{\mathcal{S}}^{m}\left( B\right) ,A_{B}\right) \leq \sum_{k\geq
m}\varphi ^{k}\left( diam\left( B\cup F_{\mathcal{S}}\left( B\right) \right)
\right)
\end{equation*}%
for every $m\in 
\mathbb{N}
$. As $\displaystyle\lim_{n\rightarrow \infty }h\left( B_{n},B\right) =0$,
using Remark \ref{MB} there exists a set $M\in P_{b}\left( X\right) $ such
that $\left( \displaystyle\cup _{n\in 
\mathbb{N}
}B_{n}\right) \cup B\subset M$. Thus, $diam\left( B_{n}\cup F_{\mathcal{S}%
}\left( B_{n}\right) \right) \leq diam\left( M\cup F_{\mathcal{S}}\left(
M\right) \right) $ for all $n\in 
\mathbb{N}
$ and $diam\left( B\cup F_{\mathcal{S}}\left( B\right) \right) \leq
diam\left( M\cup F_{\mathcal{S}}\left( M\right) \right) $. It results that%
\begin{equation*}
h\left( A_{B_{n}},A_{B}\right) \leq 2\cdot \sum_{k\geq m}\varphi ^{k}\left(
diam\left( M_{B}\cup F_{\mathcal{S}}\left( M_{B}\right) \right) \right)
+h\left( F_{\mathcal{S}}^{m}\left( B_{n}\right) ,F_{\mathcal{S}}^{m}\left(
B\right) \right)
\end{equation*}%
for every $m,n\in 
\mathbb{N}
$. Using the fact that $\displaystyle\lim_{n\rightarrow \infty }h\left(
B_{n},B\right) =0$ and $F_{\mathcal{S}}$ is continuous on bounded sets, we
infer that $\displaystyle\lim_{n\rightarrow \infty }h\left( F_{\mathcal{S}%
}^{m}\left( B_{n}\right) ,F_{\mathcal{S}}^{m}\left( B\right) \right) =0$.
From this relation and the above inequality by passing to limit as $%
n,m\rightarrow \infty $ we obtain that $\displaystyle\lim_{n\rightarrow
\infty }h\left( A_{B_{n}},A_{B}\right) =0$, so the function $F$ is
continuous.
\end{proof}

\begin{remark}
The restriction of $F$ to $P_{cl,b}\left( B\right) $ is a uniformly
continuous function, for all $B\in P_{cl,b}\left( X\right) $.
\end{remark}

\begin{proposition}
\label{aalfab}Let $\mathcal{S}=\left( \left( X,d\right) ,\left( f_{i}\right)
_{i\in I}\right) $ be a pcIIFS. Then, for every $B\in P_{b}\left( X\right) $
and $\alpha \in \Lambda \left( I\right) $, the sequence $\left( \overline{%
f_{[\alpha ]_{n}}\left( B\right) }\right) _{n}$ is convergent. If we denote
by $a_{\alpha }\left( B\right) =\displaystyle\lim_{n\rightarrow \infty }%
\overline{f_{[\alpha ]_{n}}\left( B\right) }$, then 
\begin{equation}
h\left( \overline{f_{[\alpha ]_{m}}\left( B\right) },a_{\alpha }\left(
B\right) \right) =h\left( f_{[\alpha ]_{m}}\left( B\right) ,a_{\alpha
}\left( B\right) \right) \leq \sum_{k=m}^{\infty }\varphi ^{k}\left(
diam\left( \mathcal{O}\left( B\right) \right) \right)  \label{aalfa}
\end{equation}%
for all $m\in 
\mathbb{N}
$.
\end{proposition}

\begin{proof}
Let $B\in P_{b}\left( X\right) $. We have%
\begin{equation*}
h\left( \overline{f_{[\alpha ]_{n}}\left( B\right) },\overline{f_{[\alpha
]_{n+1}}\left( B\right) }\right) =h\left( \overline{\cup _{x\in B}\left\{
f_{[\alpha ]_{n}}\left( x\right) \right\} },\overline{\cup _{x\in B}\left\{
f_{[\alpha ]_{n+1}}\left( x\right) \right\} }\right)
\end{equation*}%
\begin{equation*}
\overset{(\ref{hpeinch}),(\ref{sup})}{\leq }\sup_{x\in B}h\left( \left\{
f_{[\alpha ]_{n}}\left( x\right) \right\} ,\left\{ f_{[\alpha ]_{n+1}}\left(
x\right) \right\} \right) =\sup_{x\in B}d\left( f_{[\alpha ]_{n}}\left(
x\right) ,f_{[\alpha ]_{n+1}}\left( x\right) \right)
\end{equation*}%
\begin{equation*}
\overset{\text{Definition }\ref{iifs}\text{ }1)}{\leq }\sup_{x\in B}\varphi
^{n}\left( \sup_{i\in I}d\left( x,f_{i}\left( x\right) \right) \right)
\end{equation*}%
\begin{equation*}
\overset{\text{Definition }\ref{compfunc}\text{ }1)}{\leq }\sup_{x\in
B}\varphi ^{n}\left( diam\left( \mathcal{O}\left( x\right) \right) \right) 
\overset{\text{Definition }\ref{compfunc}\text{ }1)}{\leq }\varphi
^{n}\left( diam\left( \mathcal{O}\left( B\right) \right) \right) \text{,}
\end{equation*}%
so 
\begin{equation}
h\left( \overline{f_{[\alpha ]_{n}}\left( B\right) },\overline{f_{[\alpha
]_{n+1}}\left( B\right) }\right) =h\left( f_{[\alpha ]_{n}}\left( B\right)
,f_{[\alpha ]_{n+1}}\left( B\right) \right) \leq \varphi ^{n}\left(
diam\left( \mathcal{O}\left( B\right) \right) \right)  \label{cauchy}
\end{equation}%
for every $n\in 
\mathbb{N}
$. Applying 2) from Definition \ref{compfunc} it results that the sequence $%
\left( \overline{f_{[\alpha ]_{n}}\left( B\right) }\right) _{n}$ is Cauchy
and since the space $\left( P_{cl,b}\left( X\right) ,h\right) $ is complete
we obtain that $\left( \overline{f_{[\alpha ]_{n}}\left( B\right) }\right)
_{n}$ is convergent. Using iii) from Definition \ref{iifs1} we deduce that
there exists $a_{\alpha }\left( B\right) \in $ $P_{b}\left( X\right) $ such
that $\displaystyle\lim_{n\rightarrow \infty }\overline{f_{[\alpha
]_{n}}\left( B\right) }=a_{\alpha }\left( B\right) $.

Let $m,n\in 
\mathbb{N}
$, with $m<n$. We have%
\begin{equation*}
h\left( \overline{f_{[\alpha ]_{m}}\left( B\right) },\overline{f_{[\alpha
]_{n}}\left( B\right) }\right) \overset{(\ref{hpeinch})}{=}h\left(
f_{[\alpha ]_{m}}\left( B\right) ,f_{[\alpha ]_{n}}\left( B\right) \right)
\leq \sum_{k=m}^{n-1}h\left( f_{[\alpha ]_{k}}\left( B\right) ,f_{[\alpha
]_{k+1}}\left( B\right) \right)
\end{equation*}%
for every $m,n\in 
\mathbb{N}
$, $m<n.$ Using relation (\ref{cauchy}) we deduce%
\begin{equation*}
h\left( \overline{f_{[\alpha ]_{m}}\left( B\right) },\overline{f_{[\alpha
]_{n}}\left( B\right) }\right) \leq \sum_{k=m}^{n-1}\varphi ^{k}\left(
diam\left( \mathcal{O}\left( B\right) \right) \right)
\end{equation*}%
for every $m,n\in 
\mathbb{N}
$, $m<n$. By passing to limit as $n\rightarrow \infty $ and using the fact
that $\displaystyle\sum_{k=m}^{\infty }\varphi ^{k}\left( diam\left( 
\mathcal{O}\left( B\right) \right) \right) $ is convergent we obtain%
\begin{equation*}
h\left( \overline{f_{[\alpha ]_{m}}\left( B\right) },a_{\alpha }\left(
B\right) \right) \overset{(\ref{hpeinch})}{=}h\left( f_{[\alpha ]_{m}}\left(
B\right) ,a_{\alpha }\left( B\right) \right) \leq \sum_{k=m}^{\infty
}\varphi ^{k}\left( diam\left( \mathcal{O}\left( B\right) \right) \right)
\end{equation*}%
for every $m\in 
\mathbb{N}
$.
\end{proof}

\begin{remark}
If we take $B=\left\{ x\right\} $ we deduce that there exists a set $%
a_{\alpha }(\left\{ x\right\} )$ such that%
\begin{equation*}
\lim_{n\rightarrow \infty }\overline{f_{[\alpha ]_{n}}\left( \left\{
x\right\} \right) }=a_{\alpha }\left( \left\{ x\right\} \right) \text{.}\ 
\end{equation*}%
Since $f_{[\alpha ]_{n}}\left( \left\{ x\right\} \right) =$ $\{f_{[\alpha
]_{n}}\left( x\right) \}$ and this set has one element, it follows that $%
a_{\alpha }\left( \left\{ x\right\} \right) $ has one element denoted by $%
a_{\alpha }(x)$. In this way we define a function $a_{\alpha }:X\rightarrow
X $.
\end{remark}

\begin{lemma}
\label{aalfaB}Let $\mathcal{S}=\left( \left( X,d\right) ,\left( f_{i}\right)
_{i\in I}\right) $ be a pcIIFS. Then $a_{\alpha }\left( B\right) =\overline{%
\cup _{x\in B}\left\{ a_{\alpha }\left( x\right) \right\} }$ for every $B\in
P_{b}\left( X\right) $ and $\alpha \in \Lambda \left( I\right) $.
\end{lemma}

\begin{proof}
Let $B\in P_{b}\left( X\right) $ and $\alpha \in \Lambda \left( I\right) $.
We have 
\begin{equation*}
h\left( a_{\alpha }\left( B\right) ,\overline{\cup _{x\in B}\left\{
a_{\alpha }\left( x\right) \right\} }\right) \leq h\left( a_{\alpha }\left(
B\right) ,\overline{f_{[\alpha ]_{n}}\left( B\right) }\right) +h\left( 
\overline{f_{[\alpha ]_{n}}\left( B\right) },\overline{\cup _{x\in B}\left\{
a_{\alpha }\left( x\right) \right\} }\right)
\end{equation*}%
\begin{equation*}
\overset{(\ref{aalfa}),(\ref{hpeinch}),(\ref{sup})}{\leq }\sum_{k=n}^{\infty
}\varphi ^{k}\left( diam\left( \mathcal{O}\left( B\right) \right) \right)
+\sup_{x\in B}h\left( \left\{ f_{[\alpha ]_{n}}\left( x\right) \right\}
,\left\{ a_{\alpha }\left( x\right) \right\} \right)
\end{equation*}%
\begin{equation*}
\overset{(\ref{aalfa})\text{ }}{\leq }2\cdot \sum_{k=n}^{\infty }\varphi
^{k}\left( diam\left( \mathcal{O}\left( B\right) \right) \right)
\end{equation*}%
for every $n\in 
\mathbb{N}
$. By passing to limit as $n\rightarrow \infty $ we obtain the conclusion.
\end{proof}

\begin{lemma}
\label{unifcont} Let $\mathcal{S}=\left( \left( X,d\right) ,\left(
f_{i}\right) _{i\in I}\right) $ be a pcIIFS. For every $\alpha \in \Lambda
\left( I\right) $ and $B\in P_{cl,b}\left( X\right) $ the function $%
a_{\alpha }$ is uniformly continuous on $B$.
\end{lemma}

\begin{proof}
Let $B\in P_{cl,b}\left( X\right) $ and $\alpha \in \Lambda \left( I\right) $%
. Since%
\begin{equation*}
d\left( a_{\alpha }\left( x\right) ,f_{[\alpha ]_{m}}\left( x\right) \right)
\leq \sum_{k=m}^{\infty }\varphi ^{k}\left( diam\left( \mathcal{O}\left(
B\right) \right) \right)
\end{equation*}%
for every $m\in 
\mathbb{N}
$ and $x\in B$, we have $f_{[\alpha ]_{m}}\overset{u.c.}{\rightarrow }%
a_{\alpha }$. Applying ii) from Definition \ref{iifs1} it results that $%
f_{[\alpha ]_{m}}$ is uniformly continuous for all $m\in 
\mathbb{N}
$. We obtain that the function $a_{\alpha }$ is uniformly continuous on $B$. 
{\Huge \ }
\end{proof}

\begin{lemma}
Let $\mathcal{S}=\left( \left( X,d\right) ,\left( f_{i}\right) _{i\in
I}\right) $ be a pcIIFS. Then $f_{i}\left( a_{\alpha }\left( B\right)
\right) =a_{i\alpha }\left( B\right) $ for every $B\in P_{b}\left( X\right) $%
, $\alpha \in \Lambda \left( I\right) $ and $i\in I$.
\end{lemma}

\begin{proof}
Let $B\in P_{b}\left( X\right) $ and $\alpha \in \Lambda \left( I\right) $.
As $\displaystyle\lim_{n\rightarrow \infty }f_{[\alpha ]_{n}}\left( B\right)
=a_{\alpha }\left( B\right) $ and $f_{i}$ is a uniformly continuous function
on bounded sets for every $i\in I$, applying Proposition \ref{h} 3) we
deduce that $\displaystyle\lim_{n\rightarrow \infty }f_{i}\left( f_{[\alpha
]_{n}}\left( B\right) \right) =f_{i}\left( a_{\alpha }(B)\right) $.\textbf{\ 
}Uniqueness of the limit assures us that $a_{i\alpha }\left( B\right)
=f_{i}\left( a_{\alpha }(B)\right) $. By mathematical induction, one can
prove that%
\begin{equation}
f_{\omega }\left( a_{\alpha }\left( B\right) \right) =a_{\omega \alpha
}\left( B\right) \text{ }  \label{aialfa}
\end{equation}%
for every $\omega \in \Lambda _{n}\left( I\right) $, $n\in 
\mathbb{N}
^{\ast }$ and $B\in P_{b}\left( X\right) $.
\end{proof}

\begin{lemma}
\label{egalpeorb}Let $\mathcal{S}=\left( \left( X,d\right) ,\left(
f_{i}\right) _{i\in I}\right) $ be a pcIIFS. Then $a_{\alpha }\left(
x\right) =a_{\alpha }\left( y\right) $ \bigskip for every $x\in X$, $y\in $ $%
\overline{\mathcal{O}\left( x\right) }$ and $\alpha \in \Lambda \left(
I\right) $.
\end{lemma}

\begin{proof}
Let $y\in \overline{\mathcal{O}\left( x\right) }$ and $\alpha \in \Lambda
\left( I\right) $. As 
\begin{equation*}
\overline{\mathcal{O}\left( x\right) }=\overline{\cup _{n\in 
\mathbb{N}
}F_{S}^{n}\left( \left\{ x\right\} \right) }=\overline{\cup _{n\in 
\mathbb{N}
}\overline{\cup _{\omega \in \Lambda _{n}\left( I\right) }\left\{ f_{\omega
}\left( x\right) \right\} }}=\overline{\cup _{n\in 
\mathbb{N}
}\cup _{\omega \in \Lambda _{n}\left( I\right) }\left\{ f_{\omega }\left(
x\right) \right\} }\text{,}
\end{equation*}%
we distringuish two cases:

\begin{case}
\label{caz1}There exist $m\in 
\mathbb{N}
^{\ast }$ and $\omega \in \Lambda _{m}\left( I\right) $ such that $%
y=f_{\omega }\left( x\right) $.
\end{case}

We prove that $a_{\alpha }\left( x\right) =a_{\alpha }\left( f_{\omega
}\left( x\right) \right) $. We have%
\begin{equation*}
d\left( a_{\alpha }\left( x\right) ,a_{\alpha }\left( f_{\omega }\left(
x\right) \right) \right) \leq d\left( a_{\alpha }\left( x\right) ,f_{[\alpha
]_{n}}\left( x\right) \right)
\end{equation*}%
\begin{equation*}
+d\left( f_{[\alpha ]_{n}}\left( x\right) ,f_{[\alpha ]_{n}}\left( \left(
f_{\omega }\left( x\right) \right) \right) \right) +d\left( f_{[\alpha
]_{n}}\left( f_{\omega }\left( x\right) \right) ,a_{\alpha }\left( f_{\omega
}\left( x\right) \right) \right)
\end{equation*}%
for all $n\in 
\mathbb{N}
$. Applying relation (\ref{aalfa}) for $B=\left\{ f_{\omega }\left( x\right)
\right\} $ and then for $B=\left\{ x\right\} $ we have%
\begin{equation}
d\left( a_{\alpha }\left( x\right) ,f_{[\alpha ]_{n}}\left( x\right) \right)
\leq \sum_{k=n}^{\infty }\varphi ^{k}\left( diam\left( \mathcal{O}\left(
x\right) \right) \right)  \label{1}
\end{equation}%
and%
\begin{equation*}
d\left( a_{\alpha }\left( f_{\omega }\left( x\right) \right) ,f_{[\alpha
]_{n}}\left( f_{\omega }\left( x\right) \right) \right) \leq
\sum_{k=n}^{\infty }\varphi ^{k}\left( diam\left( \mathcal{O}\left(
f_{\omega }\left( x\right) \right) \right) \right)
\end{equation*}%
\begin{equation}
\leq \sum_{k=n}^{\infty }\varphi ^{k}\left( diam\left( \mathcal{O}\left(
x\right) \right) \right)  \label{2}
\end{equation}%
for every $n\in 
\mathbb{N}
$. Using the triangle inequality we obtain%
\begin{equation*}
d\left( f_{[\alpha ]_{n}}\left( x\right) ,f_{[\alpha ]_{n}}\left( \left(
f_{\omega }\left( x\right) \right) \right) \right) \leq d\left( f_{[\alpha
]_{n}}\left( x\right) ,f_{[\alpha ]_{n}[\omega ]_{1}}\left( x\right) \right)
\end{equation*}%
\begin{equation*}
+d\left( f_{[\alpha ]_{n}[\omega ]_{1}}\left( x\right) ,f_{[\alpha
]_{n}[\omega ]_{2}}\left( x\right) \right) +...+d\left( f_{[\alpha
]_{n}[\omega ]_{m-1}}\left( x\right) ,f_{[\alpha ]_{n}\omega }\left(
x\right) \right)
\end{equation*}%
for every $n\in 
\mathbb{N}
$ and applying (\ref{pc}) we deduce%
\begin{equation}
d\left( f_{[\alpha ]_{n}}\left( x\right) ,f_{[\alpha ]_{n}}\left( \left(
f_{\omega }\left( x\right) \right) \right) \right) \leq
\sum_{k=n}^{n+m-1}\varphi ^{k}\left( diam\left( \mathcal{O}\left( x\right)
\right) \right)  \label{3}
\end{equation}%
for every $n\in 
\mathbb{N}
$. From (\ref{1}), (\ref{2}) and (\ref{3}), taking into consideration 2)
from Definition \ref{compfunc} we conclude that $a_{\alpha }\left( x\right)
=a_{\alpha }\left( f_{\omega }\left( x\right) \right) $.

\begin{case}
\label{caz2}There exists a sequence $\left( y_{m}\right) _{m\in 
\mathbb{N}
}\subset $ $\displaystyle\cup _{n\in 
\mathbb{N}
}\displaystyle\cup _{\omega \in \Lambda _{n}\left( I\right) }\left\{
f_{\omega }\left( x\right) \right\} $ such that $\displaystyle%
\lim_{m\rightarrow \infty }d\left( y_{m},y\right) =0$.
\end{case}

In this case,%
\begin{equation*}
d\left( a_{\alpha }\left( y\right) ,a_{\alpha }\left( x\right) \right) \leq
d\left( a_{\alpha }\left( y\right) ,a_{\alpha }\left( y_{m}\right) \right)
+d\left( a_{\alpha }\left( y_{m}\right) ,a_{\alpha }\left( x\right) \right)
\end{equation*}%
for every $m\in 
\mathbb{N}
$. Using the first case we deduce that $a_{\alpha }\left( y_{m}\right)
=a_{\alpha }\left( x\right) $ for every $m\in 
\mathbb{N}
$. Hence,%
\begin{equation*}
d\left( a_{\alpha }\left( y\right) ,a_{\alpha }\left( x\right) \right) \leq
d\left( a_{\alpha }\left( y\right) ,a_{\alpha }\left( y_{m}\right) \right)
\end{equation*}%
for every $m\in 
\mathbb{N}
$ and applying Lemma \ref{unifcont} we infer that $a_{\alpha }\left(
y\right) =a_{\alpha }\left( x\right) $.
\end{proof}

\begin{theorem}
Let $\mathcal{S}=\left( \left( X,d\right) ,\left( f_{i}\right) _{i\in
I}\right) $ be a pcIIFS. Then 
\begin{equation*}
A_{B}=\overline{\cup _{\alpha \in \Lambda \left( I\right) }a_{\alpha }\left(
B\right) }=\overline{\cup _{x\in B}\cup _{\alpha \in \Lambda (I)}\left\{
a_{\alpha }\left( x\right) \right\} }
\end{equation*}%
for every $B\in P_{cl,b}\left( X\right) $.
\end{theorem}

\begin{proof}
Let us consider $B\in P_{cl,b}\left( X\right) $ and $x\in B$. We have 
\begin{equation*}
h\left( F_{\mathcal{S}}^{n}\left( \left\{ x\right\} \right) ,\overline{%
\left\{ a_{\alpha }\left( x\right) \text{ }|\text{ }\alpha \in \Lambda
\left( I\right) \right\} }\right) \overset{(\ref{hpeinch})}{=}h\left( 
\overline{\cup _{\alpha \in \Lambda \left( I\right) }f_{[\alpha ]_{n}}\left(
\left\{ x\right\} \right) },\overline{\cup _{\alpha \in \Lambda \left(
I\right) }\{a_{\alpha }\left( x\right) \}\text{ }}\right)
\end{equation*}%
\begin{equation*}
\overset{(\ref{sup})}{\leq }\sup_{\alpha \in \Lambda \left( I\right)
}h\left( \left\{ f_{[\alpha ]_{n}}\left( x\right) \right\} ,\left\{
a_{\alpha }\left( x\right) \right\} \right) =\sup_{\alpha \in \Lambda \left(
I\right) }d\left( f_{[\alpha ]_{n}}\left( x\right) ,a_{\alpha }\left(
x\right) \right)
\end{equation*}%
\begin{equation*}
\overset{(\ref{aalfa})}{\leq }\sum_{k=n}^{\infty }\varphi ^{k}\left(
diam\left( \mathcal{O}\left( B\right) \right) \right)
\end{equation*}%
for every $n\in 
\mathbb{N}
$, so 
\begin{equation}
h\left( F_{\mathcal{S}}^{n}\left( \left\{ x\right\} \right) ,\overline{%
\left\{ a_{\alpha }\left( x\right) \text{ }|\text{ }\alpha \in \Lambda
\left( I\right) \right\} }\right) \leq \sum_{k=n}^{\infty }\varphi
^{k}\left( diam\left( \mathcal{O}\left( B\right) \right) \right)  \label{Fs}
\end{equation}%
for every $n\in 
\mathbb{N}
$. We deduce 
\begin{equation*}
h\left( A_{x},\overline{\left\{ a_{\alpha }\left( x\right) \text{ }|\text{ }%
\alpha \in \Lambda \left( I\right) \right\} }\right) \leq h\left( A_{x},F_{%
\mathcal{S}}^{n}\left( \left\{ x\right\} \right) \right)
\end{equation*}%
\begin{equation*}
+h\left( F_{\mathcal{S}}^{n}\left( \left\{ x\right\} \right) ,\overline{%
\left\{ a_{\alpha }\left( x\right) \text{ }|\text{ }\alpha \in \Lambda
\left( I\right) \right\} }\right)
\end{equation*}%
\begin{equation*}
\overset{(\ref{Ax1}),(\ref{Fs})}{\leq }\sum_{k=n}^{\infty }\varphi
^{k}\left( diam\left( B\cup F_{\mathcal{S}}\left( B\right) \right) \right)
+\sum_{k=n}^{\infty }\varphi ^{k}\left( diam\left( \mathcal{O}\left(
B\right) \right) \right)
\end{equation*}%
for every $n\in 
\mathbb{N}
$. By passing to limit as $n\rightarrow \infty $ we obtain that%
\begin{equation}
A_{x}=\overline{\left\{ a_{\alpha }\left( x\right) \text{ }|\text{ }\alpha
\in \Lambda \left( I\right) \right\} }  \label{Ax}
\end{equation}%
and 
\begin{equation*}
A_{B}\overset{\text{Proposition }\ref{ABx}}{=}\overline{\cup _{x\in B}A_{x}}%
\overset{(\ref{Ax})}{=}\overline{\cup _{x\in B}\overline{\left\{ a_{\alpha
}\left( x\right) \text{ }|\text{ }\alpha \in \Lambda \left( I\right)
\right\} }}\overset{\text{Lemma }\ref{aalfaB}}{=}\overline{\cup _{\alpha \in
\Lambda (I)}a_{\alpha }\left( B\right) }\text{.}
\end{equation*}
\end{proof}

\begin{proposition}
\label{atrpeorb}Let $\mathcal{S}=\left( \left( X,d\right) ,\left(
f_{i}\right) _{i\in I}\right) $ be a pcIIFS. Then $A_{B}=A_{x}$ for every $%
x\in X$ and $B\in P_{cl,b}\left( \overline{O\left( x\right) }\right) $.
\end{proposition}

\begin{proof}
Let $x\in X$ and $B\in P_{cl,b}\left( \overline{O\left( x\right) }\right) $.
As $\displaystyle\lim_{n\rightarrow \infty }F_{\mathcal{S}}^{n}\left(
\left\{ x\right\} \right) =A_{x}$ we deduce that $\displaystyle%
\lim_{n\rightarrow \infty }\cup _{k\geq n}F_{\mathcal{S}}^{k}\left( \left\{
x\right\} \right) =A_{x}$. But $\displaystyle\cup _{k\geq n}F_{\mathcal{S}%
}^{k}\left( \left\{ x\right\} \right) =F_{\mathcal{S}}^{n}\left( \mathcal{O}%
\left( x\right) \right) $ and we obtain that $\displaystyle%
\lim_{n\rightarrow \infty }F_{\mathcal{S}}^{n}\left( \overline{\mathcal{O}%
\left( x\right) }\right) =A_{x}$. As $B$ $\subset \overline{\mathcal{O}%
\left( x\right) }$ it results that $F_{\mathcal{S}}^{n}\left( B\right)
\subset F_{\mathcal{S}}^{n}\left( \overline{\mathcal{O}\left( x\right) }%
\right) $ for every $n\in 
\mathbb{N}
$. As $\displaystyle\lim_{n\rightarrow \infty }F_{\mathcal{S}}^{n}\left(
B\right) =A_{B}$ we infer that $A_{B}\subset A_{x}$. Now, let us consider $%
y\in B$. Using Lemma \ref{egalpeorb} and relation (\ref{Ax}) it results that 
$A_{x}=A_{y}$. As $y\in B$, we have that $A_{y}\subset A_{B}$ and we deduce
that $A_{x}\subset A_{B}$. The conclusion holds from the two inclusions.
\end{proof}

\begin{proposition}
Let $\mathcal{S}=\left( \left( X,d\right) ,\left( f_{i}\right) _{i\in
I}\right) $ be a pcIIFS. Then, for every $x,y\in X$ such that $\overline{%
\mathcal{O}\left( x\right) }\cap \overline{\mathcal{O}\left( y\right) }\neq
\varnothing $ we have $A_{x}=A_{y}$. In particular, if $\overline{\mathcal{O}%
\left( x\right) }\cap \overline{\mathcal{O}\left( y\right) }\neq \varnothing 
$ for all $x,y\in X$ we have that $F_{\mathcal{S}}$ is a Picard operator.
\end{proposition}

\begin{proof}
Let $z\in \overline{\mathcal{O}\left( x\right) }\cap \overline{\mathcal{O}%
\left( y\right) }$. It results that $z\in \overline{\mathcal{O}\left(
x\right) }$. From Proposition \ref{atrpeorb} we infer that $A_{x}=A_{z}$.
Similarly, $A_{y}=A_{z}$. Therefore, we conclude that $A_{x}=A_{y}$.
\end{proof}

\begin{proposition}
Let $\mathcal{S}=\left( \left( X,d\right) ,\left( f_{i}\right) _{i\in
I}\right) $ be a pcIIFS. Then the sequence $\left( diam\left( A_{[\alpha
]_{n},x}\right) \right) _{n\in 
\mathbb{N}
}$ is convergent to $0$.
\end{proposition}

\begin{proof}
We are using the notation $C=\left\{ a_{\alpha }\left( x\right) \text{ }|%
\text{ }\alpha \in \Lambda \left( I\right) \right\} $. We have%
\begin{equation}
diam\left( A_{[\alpha ]_{n},x}\right) \overset{(\ref{Ax})}{=}diam\left(
f_{[\alpha ]_{n}}\left( \overline{C}\right) \right) =diam\left( \overline{%
f_{[\alpha ]_{n}}\left( \overline{C}\right) }\right)  \label{Aalfa}
\end{equation}%
\begin{equation*}
=diam\left( \overline{f_{[\alpha ]_{n}}\left( C\right) }\right) =diam\left(
f_{[\alpha ]_{n}}\left( C\right) \right) =\sup_{u,v\in C}d\left( f_{[\alpha
]_{n}}\left( u\right) ,f_{[\alpha ]_{n}}\left( v\right) \right)
\end{equation*}%
for every $n\in 
\mathbb{N}
$. As $u,v\in C$, we deduce that there exist $\beta $,$\gamma \in \Lambda
\left( I\right) $ such that $u=a_{\beta }\left( x\right) $ and $v=a_{\gamma
}\left( x\right) $. We obtain 
\begin{equation*}
d\left( f_{[\alpha ]_{n}}\left( u\right) ,f_{[\alpha ]_{n}}\left( v\right)
\right) =d\left( f_{[\alpha ]_{n}}\left( a_{\beta }\left( x\right) \right)
,f_{[\alpha ]_{n}}\left( a_{\gamma }\left( x\right) \right) \right) \overset{%
(\ref{aialfa})}{=}d\left( a_{[\alpha ]_{n}\beta }\left( x\right) ,a_{[\alpha
]_{n}\gamma }\left( x\right) \right)
\end{equation*}%
\begin{equation*}
\leq d\left( a_{[\alpha ]_{n}\beta }\left( x\right) ,f_{[\alpha ]_{n}}\left(
x\right) \right) +d\left( f_{[\alpha ]_{n}}\left( x\right) ,a_{[\alpha
]_{n}\gamma }\left( x\right) \right) \overset{(\ref{aalfa})}{\leq }2\cdot
\sum_{k=n}^{\infty }\varphi ^{k}\left( diam\left( \mathcal{O}\left( x\right)
\right) \right)
\end{equation*}%
for every $n\in 
\mathbb{N}
$.

By passing to limit as $n\rightarrow \infty $ we have $\displaystyle%
\lim_{n\rightarrow \infty }diam\left( A_{[\alpha ]_{n},x}\right) =0$.
\end{proof}

\begin{proposition}
Let $\mathcal{S}=\left( \left( X,d\right) ,\left( f_{i}\right) _{i\in
I}\right) $ be a pcIIFS. Then%
\begin{equation*}
a_{\alpha }\left( x\right) =\lim_{n\rightarrow \infty }A_{[\alpha ]_{n},x}
\end{equation*}%
for every $x\in X$ and $\alpha \in \Lambda \left( I\right) $.
\end{proposition}

\begin{proof}
Let $x\in X$ and $\alpha \in \Lambda \left( I\right) $. Easily, one can
prove that $\overline{A_{[\alpha ]_{n+1},x}}\subset \overline{A_{[\alpha
]_{n},x}}$ for every $n\in 
\mathbb{N}
$. Using the fact that $\displaystyle\lim_{n\rightarrow \infty }diam\left(
A_{[\alpha ]_{n},x}\right) =0$ it results that there exists an element $%
c_{\alpha }\left( x\right) \in X$ such that $\displaystyle\cap _{n\geq 1}%
\overline{A_{[\alpha ]_{n},x}}=\left\{ c_{\alpha }\left( x\right) \right\} $%
. $\displaystyle$Thus, $\displaystyle\lim_{n\rightarrow \infty }\overline{%
f_{[\alpha ]_{n}}\left( A_{x}\right) }$ $=\left\{ c_{\alpha }\left( x\right)
\right\} $. We prove that $c_{\alpha }\left( x\right) =a_{\alpha }\left(
x\right) $. We consider $\alpha =\alpha _{1}\alpha _{2}...\alpha _{n}\alpha
_{n+1}...$ and $\omega _{n}=\alpha _{n+1}\alpha _{n+2}....$ $\in \Lambda
\left( I\right) $ for every $n\in 
\mathbb{N}
^{\ast }$. As $a_{\omega _{n}}\left( x\right) \in \left\{ a_{\alpha }\left(
x\right) \text{ }|\text{ }\alpha \in \Lambda \left( I\right) \right\} $ for
every $n\in 
\mathbb{N}
^{\ast }$, taking into account relation (\ref{Ax}) we deduce $f_{[\alpha
]_{n}}\left( a_{\omega _{n}}\left( x\right) \right) \in A_{[\alpha ]_{n},x}$
for every $n\in 
\mathbb{N}
^{\ast }$. But $f_{[\alpha ]_{n}}\left( a_{\omega _{n}}\left( x\right)
\right) =a_{\alpha }\left( x\right) $ for every $n\in 
\mathbb{N}
^{\ast }$ and we obtain $a_{\alpha }\left( x\right) \in A_{[\alpha ]_{n},x}$
for all $n\in 
\mathbb{N}
^{\ast }$. Therefore, $\left\{ a_{\alpha }\left( x\right) \right\} \subset %
\displaystyle\cap _{n\geq 1}\overline{A_{[\alpha ]_{n},x}}=\left\{ c_{\alpha
}\left( x\right) \right\} $ and in conclusion $c_{\alpha }\left( x\right)
=a_{\alpha }\left( x\right) $.
\end{proof}

\begin{theorem}
\label{prcan}Let $\mathcal{S}=\left( \left( X,d\right) ,\left( f_{i}\right)
_{i\in I}\right) $ be a pcIIFS. Then the function $\Theta :\Lambda
^{t}\left( I\right) \times P_{cl,b}\left( X\right) \rightarrow
P_{cl,b}\left( X\right) $ defined by 
\begin{equation*}
\Theta \left( \alpha ,B\right) =\left\{ 
\begin{array}{c}
a_{\alpha }\left( B\right) \text{, if }\alpha \in \Lambda \left( I\right) \\ 
f_{\alpha }\left( B\right) \text{, if }\alpha \in \Lambda ^{\ast }\left(
I\right) \setminus \left\{ \lambda \right\} \\ 
B\text{, if }\alpha =\lambda%
\end{array}%
\right.
\end{equation*}%
for all $\left( \alpha ,B\right) \in \Lambda ^{t}\left( I\right) \times
P_{cl,b}\left( X\right) $ is uniformly continuous on bounded sets. In
particular, $\Theta $ is continuous.
\end{theorem}

\begin{proof}
For $\alpha \in \Lambda ^{\ast }(I)$ we make the notation $a_{\alpha }\left(
B\right) :=f_{\alpha }\left( B\right) $. Let us consider $M\in P_{cl,b}(X)$, 
$B$,$C\in P_{cl,b}(M)\subset P_{cl,b}(X)$, $\alpha $,$\beta \in \Lambda
^{t}(I)$ and $\varepsilon >0$.

Remarks:

1) If $\alpha \in \Lambda (I)$ then using (\ref{aalfa}) we have%
\begin{equation*}
h\left( \Theta \left( \alpha ,B\right) ,f_{[\alpha ]_{m}}\left( B\right)
\right) =h\left( a_{\alpha }\left( B\right) ,f_{[\alpha ]_{m}}\left(
B\right) \right) \leq \sum_{k=m}^{\infty }\varphi ^{k}\left( diam\left( 
\mathcal{O}\left( M\right) \right) \right)
\end{equation*}%
for every $m\in 
\mathbb{N}
$.

2) If $\alpha \in \Lambda ^{\ast }(I)$ then using again (\ref{aalfa}) we have%
\begin{equation*}
h\left( \Theta \left( \alpha ,B\right) ,f_{[\alpha ]_{m}}\left( B\right)
\right) =h\left( f_{\alpha }\left( B\right) ,f_{[\alpha ]_{m}}\left(
B\right) \right)
\end{equation*}%
\begin{equation*}
\leq \left\{ 
\begin{array}{c}
0\text{ if }|\alpha |\leq m \\ 
\displaystyle\sum_{k=m}^{|\alpha |}\varphi ^{k}\left( diam\left( \mathcal{O}%
\left( M\right) \right) \right)%
\end{array}%
\right. \leq \sum_{k=m}^{\infty }\varphi ^{k}\left( diam\left( \mathcal{O}%
\left( M\right) \right) \right)
\end{equation*}%
for every $m\in 
\mathbb{N}
$. Therefore, for $\alpha ,\beta \in \Lambda ^{t}(I)$ and $B,C\in
P_{cl,b}(M) $, using triangle inequality we obtain%
\begin{equation*}
h\left( \Theta \left( \alpha ,B\right) ,\Theta \left( \beta ,C\right)
\right) =h\left( a_{\alpha }\left( B\right) ,a_{\beta }\left( C\right)
\right) \leq h\left( a_{\alpha }\left( B\right) ,f_{[\alpha ]_{m}}\left(
B\right) \right)
\end{equation*}%
\begin{equation*}
+h\left( f_{[\alpha ]_{m}}\left( B\right) ,f_{[\beta ]_{m}}\left( B\right)
\right) +h\left( f_{[\beta ]_{m}}\left( B\right) ,f_{[\beta ]_{m}}\left(
C\right) \right) +h\left( f_{[\beta ]_{m}}\left( C\right) ,a_{\beta }\left(
C\right) \right)
\end{equation*}%
for every $m\in 
\mathbb{N}
$. Applying the above remarks we deduce%
\begin{equation*}
h\left( \Theta \left( \alpha ,B\right) ,\Theta \left( \beta ,C\right)
\right) \leq \sum_{k=m}^{\infty }\varphi ^{k}\left( diam\left( \mathcal{O}%
\left( M\right) \right) \right) +h\left( f_{[\alpha ]_{m}}\left( B\right)
,f_{[\beta ]_{m}}\left( B\right) \right)
\end{equation*}%
\begin{equation}
+h\left( f_{[\beta ]_{m}}\left( B\right) ,f_{[\beta ]_{m}}\left( C\right)
\right) +\sum_{k=m}^{\infty }\varphi ^{k}\left( diam\left( \mathcal{O}\left(
M\right) \right) \right)  \label{ineg1}
\end{equation}%
for every $m\in 
\mathbb{N}
$. We take $m_{\varepsilon }$ such that 
\begin{equation}
\sum_{k=m}^{\infty }\varphi ^{k}\left( diam\left( \mathcal{O}\left( M\right)
\right) \right) <\frac{\varepsilon }{3}  \label{ineg2}
\end{equation}%
for every $m\geq m_{\varepsilon }$. Let us fix $m\geq m_{\varepsilon }$. If $%
d_{c}(\alpha ,\beta )<c^{m}$ we have $[\alpha ]_{m}=[\beta ]_{m}$ which
implies 
\begin{equation}
h\left( f_{[\alpha ]_{m}}\left( B\right) ,f_{[\beta ]_{m}}\left( B\right)
\right) =0\text{.}  \label{ineg3}
\end{equation}%
As $M$ is bounded, the function $f_{[\alpha ]_{m}}$ is uniformly continuous
on $M$. So, there exists $\delta _{\varepsilon }>0$ such that for all $%
x,y\in M$ with $d(x,y)<\delta _{\varepsilon }$ we have%
\begin{equation*}
d(f_{[\beta ]_{m}}(x),f_{[\beta ]_{m}}(y))<\frac{\varepsilon }{3}\text{.}
\end{equation*}%
Hence, for $h(B,C)<\delta _{\varepsilon }$ we deduce that%
\begin{equation}
h(f_{[\beta ]_{m}}(B),f_{[\beta ]_{m}}(C))\leq \frac{\varepsilon }{3}\text{.}
\label{ineg4}
\end{equation}%
Therefore, if $d_{c}(\alpha ,\beta )<c^{m}$ and $h(B,C)<\delta _{\varepsilon
}$, using relations (\ref{ineg1}), (\ref{ineg2}), (\ref{ineg3}) and (\ref%
{ineg4}) it results 
\begin{equation*}
h\left( \Theta \left( \alpha ,B\right) ,\Theta \left( \beta ,C\right)
\right) <\frac{\varepsilon }{3}+0+\frac{\varepsilon }{3}+\frac{\varepsilon }{%
3}=\varepsilon \text{.}
\end{equation*}%
In conclusion, $\Theta $ is uniformly continuous on $M$.
\end{proof}

\begin{corollary}
\label{cor1}Let $\mathcal{S}=\left( \left( X,d\right) ,\left( f_{i}\right)
_{i\in I}\right) $ be a pcIIFS. We consider the functions $g:X\rightarrow
P_{cl,b}\left( X\right) $ defined by $g\left( x\right) =\left\{ x\right\} $
for every $x\in X$ and $h:g\left( X\right) \rightarrow X$ defined by $%
h\left( \left\{ x\right\} \right) =x.$ We now consider a function $\pi
:\Lambda \left( I\right) \times X\rightarrow $ $P_{cl,b}\left( X\right) $
given by $\pi =h\circ \Theta \circ \left( Id_{\Lambda \left( I\right)
}\times g\right) $. This function is uniformly continuous on bounded sets.
\end{corollary}

\begin{remark}
1) We note that 
\begin{equation*}
\pi \left( \alpha ,x\right) =h\circ \Theta \circ \left( Id_{\lambda }\times
g\right) \left( \alpha ,x\right) =h\circ \Theta \left( \alpha ,\left\{
x\right\} \right) =h\left( \left\{ a_{\alpha }\left( x\right) \right\}
\right) =a_{\alpha }\left( x\right) \text{.}
\end{equation*}

2) If $S$ has only one attractor, then $\pi $ is independent of $x$ and it
represents the canonical projection for a classical IIFS.
\end{remark}

\begin{corollary}
\label{cor2}Let $\mathcal{S}=\left( \left( X,d\right) ,\left( f_{i}\right)
_{i\in I}\right) $ be a pcIIFS. Let us consider the function $Id_{\Lambda
^{t}\left( I\right) }\times F:\Lambda ^{t}\left( I\right) \times
P_{cl,b}\left( X\right) \rightarrow \Lambda ^{t}\left( I\right) \times
P_{cl,b}\left( X\right) $ defined by%
\begin{equation*}
Id_{\lambda }\times F\left( \alpha ,B\right) =\left( \alpha ,A_{B}\right)
\end{equation*}%
for all $\left( \alpha ,B\right) \in \Lambda ^{t}\left( I\right) \times
P_{cl,b}\left( X\right) $. We define another function $\Psi :\Lambda
^{t}\left( I\right) \times P_{cl,b}\left( X\right) \rightarrow
P_{cl,b}\left( X\right) $ given by $\Psi =\Theta \circ Id_{\lambda }\times F$%
. The function $\Psi $ is continuous.
\end{corollary}

\begin{remark}
We note that%
\begin{equation*}
\Psi \left( \alpha ,B\right) =\left( \Theta \circ Id_{\lambda }\times
F\right) \left( \alpha ,B\right) =\Theta \left( \alpha ,A_{B}\right)
\end{equation*}%
\begin{equation*}
=\left\{ 
\begin{array}{c}
a_{\alpha }(A_{B})\text{ if }\alpha \in \Lambda \left( I\right) \\ 
f_{\alpha }\left( A_{B}\right) \text{ if }\alpha \in \Lambda ^{\ast }\left(
I\right) \diagdown \left\{ \lambda \right\} \\ 
A_{B}\text{ if }\alpha =\lambda%
\end{array}%
\right.
\end{equation*}
for all $\left( \alpha ,B\right) \in \Lambda ^{t}\left( I\right) \times
P_{cl,b}\left( X\right) $.
\end{remark}

\begin{corollary}
\label{cor3}Let $\mathcal{S}=\left( \left( X,d\right) ,\left( f_{i}\right)
_{i\in I}\right) $ be a pcIIFS. We consider the functions $g:X\rightarrow
P_{cl,b}\left( X\right) $ defined by $g\left( x\right) =\left\{ x\right\} $
for every $x\in X$ and $h:g\left( X\right) \rightarrow X$ defined by $%
h\left( \left\{ x\right\} \right) =x.$ Then the function $\pi ^{t}:\Lambda
^{t}\left( I\right) \times X\rightarrow $ $P_{cl,b}\left( X\right) $ given
by $\pi ^{t}=h\circ \Theta \circ \left( Id_{\Lambda ^{t}\left( I\right)
}\times g\right) $ is continuous.
\end{corollary}

\begin{remark}
We note that 
\begin{equation*}
\pi ^{t}\left( \alpha ,x\right) =h\circ \Theta \circ \left( Id_{\Lambda
^{t}\left( I\right) }\times g\right) \left( \alpha ,x\right) =h\circ \Theta
\left( \alpha ,\left\{ x\right\} \right)
\end{equation*}%
\begin{equation*}
=\left\{ 
\begin{array}{c}
a_{\alpha }\left( x\right) \text{, if }\alpha \in \Lambda \left( I\right) \\ 
f_{\alpha }\left( x\right) \text{, if }\alpha \in \Lambda ^{\ast }\left(
I\right) \setminus \left\{ \lambda \right\} \\ 
x\text{, if }\alpha =\lambda%
\end{array}%
\right.
\end{equation*}%
$\ \ $ \ \ \ \ for all $\left( \alpha ,x\right) \in \Lambda ^{t}\left(
I\right) \times X$.
\end{remark}

\begin{remark}
\label{rem}Let us consider $\mathcal{S}=\left( \left( X,d\right) ,\left(
f_{i}\right) _{i\in I}\right) $ a pcIIFS such that the fractal operator
associated to $\mathcal{S}$ is a Picard operator. We denote its fixed point
by $A$. For every $i\in I$ we consider the functions $F_{i}:\Lambda \left(
I\right) \rightarrow \Lambda \left( I\right) $ defined by $F_{i}(\alpha
)=i\alpha $, for all $\alpha \in \Lambda ^{t}\left( I\right) $ and $%
F_{i}^{t}:\Lambda ^{t}\left( I\right) \times X\rightarrow \Lambda ^{t}\left(
I\right) \times X$ given by $F_{i}^{t}\left( \alpha ,x\right) =\left(
i\alpha ,x\right) $ for all $\left( \alpha ,x\right) \in \Lambda ^{t}\left(
I\right) \times X$. Then the following diagrams are commutative:

\ \ \ \ \ $\ \ \ \ \ \ \ \Lambda \left( I\right) \QDATOPD. . {%
\underrightarrow{\hspace{0.29in}{\LARGE F}_{i}\hspace{0.29in}}}{{}}\Lambda
\left( I\right) $ \ \ \ \ \ \ \ \ \ \ \ \ \ \ \ \ $\Lambda ^{t}\left(
I\right) \times X\QDATOPD. . {\underrightarrow{\hspace{0.29in}{\LARGE F}%
_{i}^{t}\hspace{0.29in}}}{{}}\Lambda ^{t}\left( I\right) \times X$

\ \ $\ \ \ \ \ \ \left. 
\begin{tabular}{l}
\\ 
\\ 
\end{tabular}%
\right\downarrow \pi \qquad \hspace{0.2in}$ $\ \ \ \left. 
\begin{tabular}{l}
\\ 
\\ 
\end{tabular}%
\right\downarrow \pi $ \ \ \ \ \ \ \ \ \ \ \ \ \ \ \ \ \ $\ \ \left. 
\begin{tabular}{l}
\\ 
\\ 
\end{tabular}%
\right\downarrow \pi ^{t}\qquad \ \ \hspace{0.2in}\left. 
\begin{tabular}{l}
\\ 
\\ 
\end{tabular}%
\right\downarrow \pi ^{t}$

$\ \ \ \ \ \ \ \ \ \ \ \ A\QDATOPD. . {\text{ \ \ }\underrightarrow{\hspace{%
0.29in}{\Large f}_{i}\hspace{0.29in}}\text{ }}{{}}A\ \ \ \ \ \ \ \ \ \ \ \ \
\ \ \ \ \ \ \ \ \ \ \ \ X\QDATOPD. . {\text{ \ }\underrightarrow{\hspace{%
0.29in}{\Large f}_{i}\hspace{0.29in}\text{ }}}{{}}~X$ $\ \ $

This fact reveals that as $S_{\Lambda \left( I\right) }=\left( \Lambda
\left( I\right) ,\left( F_{i}\right) _{i\in I}\right) $ is a universal model
for the pcIIFS $S$ restricted to its fixed point, the system $S_{\Lambda
^{t}\left( I\right) }=(\Lambda ^{t}\left( I\right) \times X,\left(
F_{i}^{t}\right) _{i\in I})$ is a universal model for the pcIIFS $S$.
\end{remark}

\section{{\protect\Large Orbital }$\protect\varphi -${\protect\Large %
contractive infinite iterated function systems (oIIFSs)}}

\begin{proposition}
\label{proph}[see \cite{Str}] Let $\mathcal{S}=\left( \left( X,d\right)
,\left( f_{i}\right) _{i\in I}\right) $ be an oIIFS and for every $x\in X$
we consider the IIFS denoted by $\mathcal{S}_{x}=\left( \left( \overline{%
\mathcal{O}\left( x\right) },d\right) ,\left( f_{i}\right) _{i\in I}\right) $%
. Then $\mathcal{S}_{x}$ is a system consisting of $\varphi -$contractions
and 
\begin{equation*}
h\left( F_{\mathcal{S}_{x}}\left( B\right) ,F_{\mathcal{S}_{x}}\left(
C\right) \right) \leq \varphi \left( h\left( B,C\right) \right)
\end{equation*}%
for every $B$,$C\in P_{cl,b}\left( \overline{\mathcal{O}\left( x\right) }%
\right) $, where $F_{\mathcal{S}_{x}}$ is the fractal operator associated to
the system $\mathcal{S}_{x}$.
\end{proposition}

\begin{remark}
Applying Proposition \ref{proph} we obtain that $F_{\mathcal{S}_{x}}$ is a $%
\varphi $-contraction on $P_{cl,b}\left( \overline{\mathcal{O}\left(
x\right) }\right) $. As $\left( X,d\right) $ is a complete metric space it
results that $\left( \overline{\mathcal{O}\left( x\right) },d\right) $ is a
complete metric space. From Proposition \ref{complet} and Theorem \ref%
{thmficontr} we deduce that $F_{\mathcal{S}_{x}}$ has a unique fixed point
(which will be denoted by $A_{x}$). Moreover, the sequence $\left( F_{%
\mathcal{S}_{x}}^{n}\left( B\right) \right) _{n\in 
\mathbb{N}
}$ is convergent to $A_{x}$ and 
\begin{equation}
h\left( F_{\mathcal{S}_{x}}^{n}\left( B\right) ,A_{x}\right) \leq \varphi
^{n}\left( h\left( B,A_{x}\right) \right)  \label{aprox}
\end{equation}%
for every $n\in 
\mathbb{N}
$ and $B\in P_{cl,b}\left( \overline{\mathcal{O}\left( x\right) }\right) $.
In particular, for every $x\in X$ the sequence $\left( F_{\mathcal{S}%
_{x}}^{n}\left( \left\{ x\right\} \right) \right) _{n\in 
\mathbb{N}
}$ is convergent to $A_{x}$ and it results that $\mathcal{O}\left( x\right) $
is bounded.
\end{remark}

\begin{remark}
If $\mathcal{S}=\left( \left( X,d\right) ,\left( f_{i}\right) _{i\in
I}\right) $ is an oIIFS then $\left( diam\left( \overline{f_{[\alpha
]_{n}}\left( \mathcal{O}\left( x\right) \right) }\right) \right) _{n\in 
\mathbb{N}
}$ is convergent to $0$ for every $x\in X$. We obtain that there exists an
element $a_{\alpha }\left( x\right) \in X$ such that $\displaystyle%
\lim_{n\rightarrow \infty }\overline{f_{[\alpha ]_{n}}\left( \mathcal{O}%
\left( x\right) \right) }=\left\{ a_{\alpha }\left( x\right) \right\} $. In
this case, we have 
\begin{equation}
h\left( f_{[\alpha ]_{n}}\left( \mathcal{O}\left( x\right) \right) ,\left\{
a_{\alpha }\left( x\right) \right\} \right) \leq \varphi ^{n}\left(
diam\left( \mathcal{O}\left( x\right) \right) \right)  \label{ineq}
\end{equation}%
for all $n\in 
\mathbb{N}
$.
\end{remark}

\begin{lemma}
\lbrack see \cite{Str}] Let $\mathcal{S}=\left( \left( X,d\right) ,\left(
f_{i}\right) _{i\in I}\right) $ be an oIIFS, $x\in X$ and $a_{\alpha }\left(
x\right) \in X$ such that $\displaystyle\lim_{n\rightarrow \infty }\overline{%
f_{[\alpha ]_{n}}\left( \mathcal{O}\left( x\right) \right) }=\left\{
a_{\alpha }\left( x\right) \right\} $. Then 
\begin{equation*}
\lim_{n\rightarrow \infty }f_{[\alpha ]_{n}}\left( B\right) =\left\{
a_{\alpha }\left( x\right) \right\}
\end{equation*}%
for all $\alpha \in \Lambda \left( I\right) $ and $B\in P_{b}\left( 
\overline{\mathcal{O}\left( x\right) }\right) $. In particular, if $%
B=\left\{ y\right\} \subset \mathcal{O}\left( x\right) $, we have $%
\displaystyle\lim_{n\rightarrow \infty }f_{[\alpha ]_{n}}\left( y\right)
=a_{\alpha }\left( x\right) $ and 
\begin{equation}
d\left( f_{[\alpha ]_{n}}\left( y\right) ,a_{\alpha }\left( x\right) \right)
\leq \varphi ^{n}\left( diam\left( \mathcal{O}\left( x\right) \right) \right)
\label{aalfa2}
\end{equation}%
for every $n\in 
\mathbb{N}
$.
\end{lemma}

We consider $B\in P_{b}\left( X\right) $ and we define 
\begin{equation*}
\hat{C}_{B}=\left\{ f:B\rightarrow X\text{ }|\text{ }f\text{ is continuous
and bounded}\right\} \text{.}
\end{equation*}
\ 

On $\hat{C}_{B}$ we consider the metric $d_{u}$ defined by $d_{u}\left(
f,g\right) =\displaystyle\sup_{x\in B}d\left( f\left( x\right) ,g\left(
x\right) \right) $. For every $i\in I$ we define the function $\hat{F}_{i}:%
\hat{C}_{B}\rightarrow \hat{C}_{B}$ given by $\hat{F}_{i}\left( f\right)
=f_{i}\circ f$ \ for every $f\in \hat{C}_{B}$. The orbit of a function $h\in 
\hat{C}_{B}$ is defined by $\mathcal{O}\left( \left\{ h\right\} \right) =%
\displaystyle\cup _{n\geq 0}\hat{F}_{\mathcal{S}}^{n}\left( \left\{
h\right\} \right) $, where $\hat{F}_{\mathcal{S}}$ is the fractal operator
associated to the system $\hat{S}=\left( (\hat{C}_{B},d_{u}),(\hat{F}%
_{i})_{i\in I}\right) $.

\begin{proposition}
Let $\mathcal{S}=\left( \left( X,d\right) ,\left( f_{i}\right) _{i\in
I}\right) $ be an oIIFS. Then the system $\hat{S}=\left( (\hat{C}%
_{B},d_{u}),(\hat{F}_{i})_{i\in I}\right) $ is an oIIFS.
\end{proposition}

\begin{proof}
Let $h\in \hat{C}_{B}$ and $g$ $\in $ $\mathcal{O}\left( \left\{ h\right\}
\right) $.

\begin{claim}
$g\left( x\right) \in \mathcal{O}\left( \left\{ h\left( x\right) \right\}
\right) $ for every $g$ $\in $ $\mathcal{O}\left( \left\{ h\right\} \right) $
and every $x\in B$.
\end{claim}

Justification: As $g\in \mathcal{O}\left( \left\{ h\right\} \right) $, we
have 
\begin{equation*}
g\in \displaystyle\cup _{n\geq 0}\hat{F}_{\mathcal{S}}^{n}\left( \left\{
h\right\} \right) =\displaystyle\cup _{n\geq 0}\overline{\cup _{\alpha \in
\Lambda _{n}\left( I\right) }\hat{F}_{\alpha }\left( \left\{ h\right\}
\right) }=\displaystyle\cup _{n\geq 0}\overline{\cup _{\alpha \in \Lambda
_{n}\left( I\right) }\left\{ f_{\alpha }\circ h\right\} }\text{.}
\end{equation*}%
Thus, there exists $n_{0}\in 
\mathbb{N}
$ such that $g\in \overline{\cup _{\alpha \in \Lambda _{n_{0}}\left(
I\right) }\left\{ f_{\alpha }\circ h\right\} }$. Here we distinguish two
cases:

1) $g\in \cup _{\alpha \in \Lambda _{n_{0}}\left( I\right) }\left\{
f_{\alpha }\circ h\right\} $. It results that there exists $\alpha \in
\Lambda _{n_{0}}\left( I\right) $ such that $g=f_{\alpha }\circ h$. Hence, $%
g\left( x\right) =f_{\alpha }\left( h\left( x\right) \right) \subset
F_{S}^{n_{0}}\left( \left\{ h\left( x\right) \right\} \right) \in \mathcal{O}%
\left( \left\{ h\left( x\right) \right\} \right) $. So, $g\left( x\right)
\in \mathcal{O}\left( \left\{ h\left( x\right) \right\} \right) $.

2) $g\notin \cup _{\alpha \in \Lambda _{n_{0}}\left( I\right) }\left\{
f_{\alpha }\circ h\right\} $. In this case there exists a sequence $\left(
g_{m}\right) _{m\in 
\mathbb{N}
}\subset \cup _{\alpha \in \Lambda _{n_{0}}\left( I\right) }\left\{
f_{\alpha }\circ h\right\} $ such that $\displaystyle\lim_{m\rightarrow
\infty }d_{u}\left( g_{m},g\right) =0$. We have that for every $m\in 
\mathbb{N}
$, there is $\alpha _{m}\in \Lambda _{n_{0}}\left( I\right) $ such that $%
g_{m}=f_{\alpha _{m}}\circ h$. We have $g_{m}\left( x\right) =f_{\alpha
_{m}}\left( h\left( x\right) \right) \in F_{S}^{n_{0}}\left( \left\{ h\left(
x\right) \right\} \right) $ and using the fact that $F_{S}^{n_{0}}\left(
\left\{ h\left( x\right) \right\} \right) $ is a closed set, we deduce that $%
g\left( x\right) \in F_{S}^{n_{0}}\left( \left\{ h\left( x\right) \right\}
\right) \in \mathcal{O}\left( \left\{ h\left( x\right) \right\} \right) $.
So, $g\left( x\right) \in \mathcal{O}\left( \left\{ h\left( x\right)
\right\} \right) $.

\begin{claim}
$\hat{S}=\left( (\hat{C}_{B},d_{u}),(\hat{F}_{i})_{i\in I}\right) $ is an
oIIFS.
\end{claim}

Justification: Let us consider $h\in \hat{C}_{B}$ \ and $f,g\in \mathcal{O}%
\left( \left\{ h\right\} \right) $. Then $f\left( x\right) ,g\left( x\right)
\in \mathcal{O}\left( \left\{ h\left( x\right) \right\} \right) $ for all $%
x\in B$. Using this we have%
\begin{equation*}
d_{u}\left( \hat{F}_{i_{1}...i_{n}}\left( f\right) ,\hat{F}%
_{i_{1}...i_{n}}\left( g\right) \right) =d_{u}\left( f_{i_{1}...i_{n}}\left(
f\right) ,f_{i_{1}...i_{n}}\left( g\right) \right)
\end{equation*}%
\begin{equation*}
=\sup_{x\in B}d\left( f_{i_{1}...i_{n}}\left( f\left( x\right) \right)
,f_{i_{1}...i_{n}}\left( g\left( x\right) \right) \right) \overset{(\ref{o})}%
{\leq }\sup_{x\in B}\varphi ^{n}\left( d\left( f\left( x\right) ,g\left(
x\right) \right) \right)
\end{equation*}%
\begin{equation*}
\overset{\text{Definition }\ref{compfunc}\text{ }1)}{\leq }\varphi
^{n}\left( \sup_{x\in B}d\left( f\left( x\right) ,g\left( x\right) \right)
\right) =\varphi ^{n}\left( d_{u}\left( f,g\right) \right)
\end{equation*}%
for every $n\in 
\mathbb{N}
$. We deduce that $\hat{S}=\left( (\hat{C}_{B},d_{u}),(\hat{F}_{i})_{i\in
I}\right) $ is an oIIFS.
\end{proof}

\begin{remark}
The set $\mathcal{O}\left( \left\{ h\right\} \right) $ is bounded for every $%
h\in \hat{C}_{B}$.
\end{remark}

\begin{lemma}
Let $\mathcal{S}=\left( \left( X,d\right) ,\left( f_{i}\right) _{i\in
I}\right) $ be an oIIFS. Then $\mathcal{O}\left( B\right) $ is bounded for
every $B\in P_{b}\left( X\right) $.
\end{lemma}

\begin{proof}
Let $B\in P_{b}\left( X\right) $ and $\alpha \in \Lambda ^{\ast }\left(
I\right) $. We have $\hat{F}_{\alpha }\left( Id_{B}\right) =f_{\alpha }\circ
Id_{B}\in \mathcal{O}\left( \left\{ Id_{B}\right\} \right) $. As $\mathcal{O}%
\left( \left\{ Id_{B}\right\} \right) $ is bounded we deduce that $%
\displaystyle\cup _{\alpha \in \Lambda ^{\ast }\left( I\right) }\left\{
f_{\alpha }\circ Id_{B}\right\} $ is bounded, so $\displaystyle\cup _{\alpha
\in \Lambda ^{\ast }\left( I\right) }\left\{ f_{\alpha }\circ Id_{B}\left(
B\right) \right\} $ is bounded. But $\overline{\cup _{\alpha \in \Lambda
^{\ast }\left( I\right) }\left\{ f_{\alpha }\circ Id_{B}\left( B\right)
\right\} }=\overline{\mathcal{O}\left( B\right) }$ and we infer that $%
\mathcal{O}\left( B\right) $ is bounded.
\end{proof}

\begin{theorem}
Let $\mathcal{S}=\left( \left( X,d\right) ,\left( f_{i}\right) _{i\in
I}\right) $ be an oIIFS and $F_{\mathcal{S}}$ the associated fractal
operator. Then $F_{\mathcal{S}}$ is a weakly Picard operator. Moreover,%
\begin{equation}
h\left( F_{\mathcal{S}}^{n}\left( B\right) ,A_{B}\right) \leq \varphi
^{n}\left( diam\left( \overline{\mathcal{O}\left( B\right) }\right) \right)
\label{AB2}
\end{equation}%
for all $n\in 
\mathbb{N}
$ and $B\in P_{cl,b}\left( X\right) $, where $A_{B}=\overline{\cup _{x\in
B}A_{x}}$.
\end{theorem}

\begin{proof}
Let us consider $B\in P_{cl,b}\left( X\right) $ and $A_{B}=\overline{\cup
_{x\in B}A_{x}}$. We have%
\begin{equation*}
h\left( F_{\mathcal{S}}^{n}\left( \left\{ x\right\} \right) ,A_{x}\right) 
\overset{(\ref{aprox})}{\leq }\varphi ^{n}\left( h\left( \left\{ x\right\}
,A_{x}\right) \right) \leq \varphi ^{n}\left( diam\left( \mathcal{O}\left(
x\right) \right) \right) \text{,}
\end{equation*}%
so%
\begin{equation}
h\left( F_{\mathcal{S}}^{n}\left( \left\{ x\right\} \right) ,A_{x}\right)
\leq \varphi ^{n}\left( diam\left( \mathcal{O}\left( x\right) \right) \right)
\label{Ax2}
\end{equation}%
for every $n\in 
\mathbb{N}
$. We deduce%
\begin{equation*}
h\left( F_{\mathcal{S}}^{n}\left( B\right) ,A_{B}\right) =h\left( \cup
_{x\in B}F_{\mathcal{S}}^{n}\left( \left\{ x\right\} \right) ,\cup _{x\in
B}A_{x}\right) \overset{(\ref{sup})}{\leq }\sup_{x\in B}h\left( F_{\mathcal{S%
}}^{n}\left( \left\{ x\right\} \right) ,A_{x}\right)
\end{equation*}%
\begin{equation*}
\overset{(\ref{Ax2})}{\leq }\sup_{x\in B}\varphi ^{n}\left( diam\left( 
\overline{\mathcal{O}\left( x\right) }\right) \right) \overset{\text{%
Definition }\ref{compfunc}\text{ 1)}}{\leq }\varphi ^{n}\left( diam\left( 
\overline{\mathcal{O}\left( B\right) }\right) \right)
\end{equation*}%
for every $n\in 
\mathbb{N}
$. Therefore, 
\begin{equation*}
h\left( F_{\mathcal{S}}^{n}\left( B\right) ,A_{B}\right) \leq \varphi
^{n}\left( diam\left( \overline{\mathcal{O}\left( B\right) }\right) \right)
\end{equation*}%
for every $n\in 
\mathbb{N}
$.

By passing to limit as $n\rightarrow \infty $ we obtain that $\displaystyle%
\lim_{n\rightarrow \infty }h\left( F_{\mathcal{S}}^{n}\left( B\right)
,A_{B}\right) =0$. Using ii) from Definition \ref{iifs1} we deduce that $F_{%
\mathcal{S}}$ is uniformly continuous on bounded subsets of $X$. Hence, $F_{%
\mathcal{S}}\left( A_{B}\right) =A_{B}$.
\end{proof}

\begin{remark}
\label{atrpeorb2}If $\mathcal{S}=\left( \left( X,d\right) ,\left(
f_{i}\right) _{i\in I}\right) $ is an oIIFS, then $A_{B}=A_{x}$ for every $%
x\in X$ and $B\in P_{cl,b}\left( \overline{\mathcal{O}\left( x\right) }%
\right) $.
\end{remark}

\begin{proposition}
Let $\mathcal{S}=\left( \left( X,d\right) ,\left( f_{i}\right) _{i\in
I}\right) $ be an oIIFS. Then the function $F:P_{cl,b}\left( X\right)
\rightarrow P_{cl,b}\left( X\right) $ defined by $F\left( B\right) =A_{B}$
for all $B\in P_{cl,b}\left( X\right) $ is continuous.
\end{proposition}

\begin{proof}
The proof is similar with the one used for a pcIIFS, taking into
consideration relation (\ref{AB2}).
\end{proof}

\begin{proposition}
Let $\mathcal{S}=\left( \left( X,d\right) ,\left( f_{i}\right) _{i\in
I}\right) $ be an oIIFS. Then for all $B\in P_{b}\left( X\right) $ and $%
\alpha \in \Lambda \left( I\right) $ the sequence $\left( f_{[\alpha
]_{n}}\left( B\right) \right) _{n\in 
\mathbb{N}
}$ is convergent. If we denote its limit by $a_{\alpha }\left( B\right) $ we
have%
\begin{equation}
h\left( f_{[\alpha ]_{n}}\left( B\right) ,a_{\alpha }\left( B\right) \right)
\leq \varphi ^{n}\left( diam\left( \mathcal{O}\left( B\right) \right) \right)
\label{aalfab2}
\end{equation}%
for all $n\in 
\mathbb{N}
$.
\end{proposition}

\begin{proof}
Let us consider $B\in P_{b}\left( X\right) $ and $\alpha \in \Lambda \left(
I\right) $. We use the notation $a_{\alpha }\left( B\right) =\overline{\cup
_{x\in B}\left\{ a_{\alpha }\left( x\right) \right\} }$. Then%
\begin{equation*}
h\left( f_{[\alpha ]_{n}}\left( B\right) ,a_{\alpha }\left( B\right) \right)
=h\left( \cup _{x\in B}\left\{ f_{[\alpha ]_{n}}\left( x\right) \right\}
,\cup _{x\in B}\left\{ a_{\alpha }\left( x\right) \right\} \right)
\end{equation*}%
\begin{equation*}
\overset{(\ref{sup})}{\leq }\sup_{x\in B}d\left( f_{[\alpha ]_{n}}\left(
x\right) ,a_{\alpha }\left( x\right) \right) \leq \sup_{x\in B}h\left(
f_{[\alpha ]_{n}}\left( \mathcal{O}\left( x\right) \right) ,\left\{
a_{\alpha }\left( x\right) \right\} \right)
\end{equation*}%
for all $n\in 
\mathbb{N}
$. \ We deduce%
\begin{equation*}
h\left( f_{[\alpha ]_{n}}\left( B\right) ,a_{\alpha }\left( B\right) \right) 
\overset{(\ref{ineq})}{\leq }\sup_{x\in B}\varphi ^{n}\left( diam\left( 
\mathcal{O}\left( x\right) \right) \right) \overset{\text{Definition }\ref%
{compfunc}\text{ 1)}}{\leq }\varphi ^{n}\left( diam\left( \mathcal{O}\left(
B\right) \right) \right)
\end{equation*}%
for all $n\in 
\mathbb{N}
$. So,%
\begin{equation*}
h\left( f_{[\alpha ]_{n}}\left( B\right) ,a_{\alpha }\left( B\right) \right)
\leq \varphi ^{n}\left( diam\left( \mathcal{O}\left( B\right) \right) \right)
\end{equation*}%
for all $n\in 
\mathbb{N}
$. Therefore, we obtain that the sequence $\left( f_{[\alpha ]_{n}}\left(
B\right) \right) _{n\in 
\mathbb{N}
}$ is convergent to $a_{\alpha }\left( B\right) $.
\end{proof}

\begin{lemma}
\label{unifcont2}Let $\mathcal{S}=\left( \left( X,d\right) ,\left(
f_{i}\right) _{i\in I}\right) $ be an oIIFS. For every $\alpha \in \Lambda
\left( I\right) $ and $B\in P_{cl,b}\left( X\right) $, the function $%
a_{\alpha }$ is uniformly continuous on $B$.
\end{lemma}

\begin{proof}
The proof is similar with the one used in the case of a pcIIFS, taking into
account relation (\ref{aalfa2}).
\end{proof}

\begin{lemma}
Let $\mathcal{S}=\left( \left( X,d\right) ,\left( f_{i}\right) _{i\in
I}\right) $ be an oIIFS. Then%
\begin{equation*}
f_{i}\left( a_{\alpha }\left( B\right) \right) =a_{i\alpha }\left( B\right)
\end{equation*}%
for every $B\in P_{b}\left( X\right) $, $\alpha \in \Lambda \left( I\right) $
and $i\in I$.
\end{lemma}

\begin{proof}
The proof is similar with the one used in the case of a pcIIFS.
\end{proof}

\begin{theorem}
Let $\mathcal{S}=\left( \left( X,d\right) ,\left( f_{i}\right) _{i\in
I}\right) $ be an oIIFS. Then%
\begin{equation*}
A_{B}=\overline{\cup _{\alpha \in \Lambda \left( I\right) }a_{\alpha }\left(
B\right) }=\overline{\cup _{x\in B}\cup _{\alpha \in \Lambda (I)}\left\{
a_{\alpha }\left( x\right) \right\} }
\end{equation*}%
for every $B\in P_{cl,b}\left( X\right) $.
\end{theorem}

\begin{proof}
Using a technique similar with the one used in the case of a pcIIFS, one can
prove that 
\begin{equation}
h\left( F_{\mathcal{S}}^{n}\left( \left\{ x\right\} \right) ,\overline{%
\left\{ a_{\alpha }\left( x\right) \text{ }|\text{ }\alpha \in \Lambda
\left( I\right) \right\} }\right) \overset{(\ref{aalfab2})}{\leq }\varphi
^{n}\left( diam\left( \mathcal{O}\left( x\right) \right) \right)
\label{thm4.2.1}
\end{equation}%
and%
\begin{equation*}
h\left( A_{x},\overline{\left\{ a_{\alpha }\left( x\right) \text{ }|\text{ }%
\alpha \in \Lambda \left( I\right) \right\} }\right) \leq h\left( A_{x},F_{%
\mathcal{S}}^{n}\left( \left\{ x\right\} \right) \right)
\end{equation*}%
\begin{equation*}
+h\left( F_{\mathcal{S}}^{n}\left( \left\{ x\right\} \right) ,\overline{%
\left\{ a_{\alpha }\left( x\right) \text{ }|\text{ }\alpha \in \Lambda
\left( I\right) \right\} }\right) \overset{(\ref{Ax2}),\left( \ref{thm4.2.1}%
\right) }{\leq }2\cdot \varphi ^{n}\left( diam\left( \mathcal{O}\left(
x\right) \right) \right)
\end{equation*}%
for every $n\in 
\mathbb{N}
$. Hence, 
\begin{equation}
A_{x}=\overline{\left\{ a_{\alpha }\left( x\right) \text{ }|\text{ }\alpha
\in \Lambda \left( I\right) \right\} }  \label{relAx}
\end{equation}%
and we obtain that%
\begin{equation*}
A_{B}=\overline{\cup _{x\in B}\overline{\left\{ a_{\alpha }\left( x\right) 
\text{ }|\text{ }\alpha \in \Lambda \left( I\right) \right\} }}=\overline{%
\cup _{x\in B}\cup _{\alpha \in \Lambda (I)}\left\{ a_{\alpha }\left(
x\right) \right\} }
\end{equation*}%
\begin{equation*}
=\overline{\cup _{\alpha \in \Lambda (I)}\overline{\left\{ a_{\alpha }\left(
x\right) \text{ }|\text{ }x\in B\right\} }}=\overline{\cup _{\alpha \in
\Lambda (I)}a_{\alpha }\left( B\right) }\text{.}
\end{equation*}
\end{proof}

\begin{lemma}
\label{egalpeorb2}Let $\mathcal{S}=\left( \left( X,d\right) ,\left(
f_{i}\right) _{i\in I}\right) $ be an oIIFS. Then $a_{\alpha }\left(
x\right) =a_{\alpha }\left( y\right) $ for all $y\in \overline{\mathcal{O}%
\left( x\right) }$. Moreover, for every $x,y\in X$ such that $\overline{%
\mathcal{O}\left( x\right) }\cap \overline{\mathcal{O}\left( y\right) }\neq
\varnothing $ we have $A_{x}=A_{y}$.
\end{lemma}

\begin{proof}
Using Remark \ref{atrpeorb2} and relation (\ref{relAx}) we deduce the
conclusion.
\end{proof}

\begin{proposition}
Let $\mathcal{S}=\left( \left( X,d\right) ,\left( f_{i}\right) _{i\in
I}\right) $ be an oIIFS. Then the sequence $\left( diam\left( A_{[\alpha
]_{n},x}\right) \right) _{n\in 
\mathbb{N}
}$ is convergent to $0$.
\end{proposition}

\begin{proof}
Using (\ref{ineq}) and a technique similar with the one used in the case of
a pcIIFS, one can prove that $\left( diam\left( A_{[\alpha ]_{n},x}\right)
\right) _{n\in 
\mathbb{N}
}$ is convergent to $0$.
\end{proof}

\begin{proposition}
Let $\mathcal{S}=\left( \left( X,d\right) ,\left( f_{i}\right) _{i\in
I}\right) $ be an oIIFS. Then%
\begin{equation*}
a_{\alpha }\left( x\right) =\lim_{n\rightarrow \infty }A_{[\alpha ]_{n},x}
\end{equation*}%
for every $x\in X$ and $\alpha \in \Lambda \left( I\right) $.
\end{proposition}

\begin{proof}
The proof is similar with the one used in the case of a pcIIFS.
\end{proof}

\begin{theorem}
Let $\mathcal{S}=\left( \left( X,d\right) ,\left( f_{i}\right) _{i\in
I}\right) $ be an oIIFS. Then the function $\Theta :\Lambda ^{t}\left(
I\right) \times P_{cl,b}\left( X\right) \rightarrow P_{cl,b}\left( X\right) $
defined by 
\begin{equation*}
\Theta \left( \alpha ,B\right) =\left\{ 
\begin{array}{c}
a_{\alpha }\left( B\right) \text{, if }\alpha \in \Lambda \left( I\right) \\ 
f_{\alpha }\left( B\right) \text{, if }\alpha \in \Lambda ^{\ast }\left(
I\right) \setminus \left\{ \lambda \right\} \\ 
B\text{, if }\alpha =\lambda%
\end{array}%
\right.
\end{equation*}%
for all $\left( \alpha ,B\right) \in \Lambda ^{t}\left( I\right) \times
P_{cl,b}\left( X\right) $ is uniformly continuous on bounded sets.
\end{theorem}

\begin{proof}
Let $B\in P_{cl,b}\left( X\right) $. We have 
\begin{equation*}
h\left( a_{\alpha }\left( B\right) ,f_{[\alpha ]_{m}}\left( B\right) \right)
\leq \varphi ^{m}\left( diam\left( \mathcal{O}\left( B\right) \right) \right)
\end{equation*}%
for every $m\in 
\mathbb{N}
$, $\left( \alpha ,B\right) \in \Lambda \left( I\right) \times
P_{cl,b}\left( X\right) $ and%
\begin{equation*}
h\left( f_{\alpha }\left( B\right) ,f_{[\alpha ]_{m}}\left( B\right) \right)
\leq \left\{ 
\begin{array}{c}
0\text{ if }|\alpha |\leq m \\ 
\varphi ^{m}\left( diam\left( \mathcal{O}\left( B\right) \right) \right)%
\end{array}%
\right. \leq \varphi ^{m}\left( diam\left( \mathcal{O}\left( B\right)
\right) \right)
\end{equation*}%
for every $m\in 
\mathbb{N}
$ and $\left( \alpha ,B\right) \in \Lambda ^{\ast }\left( I\right) \times
P_{cl,b}\left( X\right) $. Using these relations and a technique similar
with the one used in Theorem \ref{prcan}, one can prove that $\Theta $ is
uniformly continuous.
\end{proof}

\begin{remark}
The results of Corollaries \ref{cor1}, \ref{cor2} and \ref{cor3} remain true
in the case of an oIIFS.
\end{remark}

\bigskip

\bigskip

\end{document}